\documentclass[a4paper,11pt]{amsart}
\usepackage{amssymb}
\usepackage[arrow,matrix]{xy}
\usepackage{hyperref}

\theoremstyle{ams}
\newtheorem{theorem}{Theorem}[section]
\newtheorem{proposition}[theorem]{Proposition}
\newtheorem{lemma}[theorem]{Lemma}
\newtheorem{corollary}[theorem]{Corollary}

\theoremstyle{definition}

\newtheorem{definition}[theorem]{Definition}
\newtheorem{remark}[theorem]{Remark}

\newtheorem{example}[theorem]{Example}

\newcommand{\C}{\mathbb{C}}

\newcommand{\R}{\mathbb{R}}
\newcommand{\Z}{\mathbb{Z}}

\newcommand{\CP}{\mathbb{C}P}

\newcommand{\red}{\mathrm{red}}

\begin{document}
\title[Semifree circle actions on symplectic manifolds]
{Semifree Hamiltonian circle actions on 6-dimensional symplectic manifolds with non-isolated fixed point set}

\author[Y. Cho]{Yunhyung Cho}
\address{School of mathematics, KIAS, 85 Heogiro, Dongdaemun Gu, Seoul 130-722, Korea}
\email{yhcho@kias.re.kr}

\author[T. Hwang]{Taekgyu Hwang}
\address{Department of Mathematical Sciences, KAIST, 335 Gwahangno, Yu-sung Gu, Daejeon 305-701, Korea}
\email{hwangtaekkyu@kaist.ac.kr}

\author[D. Y. Suh]{Dong Youp Suh}
\address{Department of Mathematical Sciences, KAIST, 335 Gwahangno, Yu-sung Gu, Daejeon 305-701, Korea}
\email{dysuh@math.kaist.ac.kr}

\thanks{The third author was supported in part by Basic Science Research Program through the National Research Foundation of Korea(NRF) grant funded by the Ministry of Education(2013R1A1A2007780)}


\maketitle

\begin{abstract}
	Let $(M, \omega)$ be a $6$-dimensional closed symplectic manifold with a symplectic $S^1$-action with $M^{S^1} \neq \emptyset$ and $\dim M^{S^1} \leq 2$. Assume that $\omega$ is integral with a generalized moment map $\mu$. We first prove that the action is Hamiltonian if and only if $b_2^+(M_{\red})=1$, where $M_{\red}$ is any reduced space with respect to $\mu$. It means that if the action is non-Hamiltonian, then $b_2^+(M_{\red}) \geq 2$. Secondly, we focus on the case when the action is semifree and Hamiltonian. We prove that if $M^{S^1}$ consists of surfaces, then the number $k$ of fixed surfaces with positive genera is at most four. In particular, if the extremal fixed surfaces are spheres, then $k$ is at most one. Finally, we prove that $k \neq 2$ and we construct some examples of 6-dimensional semifree Hamiltonian $S^1$-manifolds such that $M^{S^1}$ contains $k$ surfaces of positive genera for $k = 0$ and $4$. Examples with $k=1$ and $3$ were given in \cite{L2}.

\end{abstract}

\section{Introduction}

	Let $(M^{2n},\omega)$ be a $2n$-dimensional closed symplectic manifold with a symplectic $S^1$ action. Many mathematicians tried to find some conditions on $M$ which make a symplectic circle action Hamiltonian. One easy condition is that $M$ is simply connected. Then any closed $1$-form is exact, so any symplectic circle action is Hamiltonian. In 1959, Frankel discovered a condition in terms of the fixed point set in the K\"{a}hler category. The following theorem will be referred as Frankel's theorem throughout. Since any Hamiltonian action on the compact space has a fixed point, the only if part is trivial.

\begin{theorem}\cite{F}
	Assume that $\omega$ is a K\"{a}hler form and the circle action is holomorphic. Then the symplectic $S^1$ action is Hamiltonian if and only if it has fixed points.
\end{theorem}

	Ono extended Frankel's theorem to the case when $(M^{2n},\omega)$ is symplectic and satisfies the Lefschetz condition, i.e., $\wedge \, \omega^{n-1}$ induces an isomorphism from $H^1(M,\R)$ to $H^{2n-1}(M,\R)$. K\"{a}hler manifolds are well-known to satisfy the Lefschetz condition. See \cite{McS} p.154 or \cite{O}.

	Unfortunately, Frankel's theorem does not extend to the symplectic category in general. In 1988, McDuff~\cite{MD} proved that Frankel's theorem holds for $4$-dimensional symplectic manifolds. But she also found a $6$-dimensional manifold with a symplectic non-Hamiltonian $S^1$ action whose fixed point set is not empty. In her example, all fixed components are $2$-tori and all reduced spaces are diffeomorphic to the $4$-torus. Other such kind of examples are still unknown. See also \cite{K}.

	As McDuff showed, the existence of fixed points does not guarantee that the symplectic circle action is Hamiltonian. But Frankel's theorem can be extended to the symplectic category under some additional conditions . For example, Tolman and Weitsman~\cite{TW} considered semifree circle actions on closed symplectic manifolds with finite fixed points. They proved Frankel's theorem using the localization theorem for equivariant cohomology. In~\cite{Go} and~\cite{LO}, more about Frankel's theorem is discussed under some other conditions.

	In this paper, we assume that the cohomology class $[\omega]$ is integral so that there exists a generalized moment map $\mu : M \rightarrow S^1$. The generalized moment map was introduced in~\cite{MD}. We give the definition by Audin~\cite{Au} in Section~\ref{background}. The symplectic reduction $M_t:= \mu^{-1}(t)/S^1$ carries natural orientation induced from the symplectic form. Recall that $b_2^+$ of an oriented closed $4$-manifold is defined to be the maximal dimension of the subspace of $H^2(M, \R)$ on which the cup product is positive definite. The reduced space may be an orbifold, but Poincar\'{e} duality for orbifolds is enough to define~$b_2^+$. We state our main theorems.

\begin{theorem}\label{main-}
	Let $(M,\omega)$ be a $6$-dimensional closed symplectic $S^1$-manifold with generalized moment map $\mu : M \rightarrow S^1$. Assume that the fixed point set is not empty and the dimension of each component is at most $2$. Then the action is Hamiltonian if and only if $b_2^+(M_{\xi}) = 1$ for any regular value~$\xi$ of~$\mu$.
\end{theorem}

	Theorem~\ref{main-} implies that if the action is non-Hamiltonian, then $b_2^+(M_{\xi}) \geq 2$ for some, hence every, $\xi \in S^1$. We prove Theorem~\ref{main-} in Section~\ref{proof of 1.2} using a result of~\cite{Lin}. In the following theorem, we assume that the action is semifree. An $S^1$ action is called \textit{semifree} if it is free outside the fixed point set.

\begin{theorem} \label{main}
	Let $(M, \omega)$ be a $6$-dimensional closed symplectic $S^1$-manifold with generalized moment map. Suppose the action is semifree and the fixed components are all surfaces, so that all reduced spaces are diffeomorphic as smooth manifolds. Then the reduced space $M_{\red}$ is diffeomorphic to an $S^2$-bundle over a compact Riemann surface $\Sigma_g$ of genus~$g$ if and only if the action is Hamiltonian. Moreover, if the action is Hamiltonian, the number of fixed surfaces with positive genera is at most four and cannot be equal to two. In particular, if the extremal fixed surfaces are spheres, then this number is at most one. If the number is four, then all genera of the four fixed surfaces are equal to~$g$.
\end{theorem}

	Note that if the action is semifree and $M^{S^1}$ is isolated, then $M \cong S^2 \times S^2 \times S^2$ and $|M^{S^1}| = 8$ as in~\cite{TW}. On the other hand if $M^{S^1}$ consists of surfaces, we will see in Section~\ref{construction of example} that the number of fixed components can be arbitrary. But the number of fixed surfaces with positive genera is bounded by four by Theorem~\ref{main}.

	Since any ruled surface has $b_2^+ = 1$, Theorem~\ref{main-} implies the first statement of Theorem~\ref{main}. However, to get the second part of Theorem~\ref{main}, we need a different approach for the proof of the first part of Theorem~\ref{main}. We need to study the change of the reduced symplectic form $\omega_t$ and the Euler class~$e$ on $H^2(M_{\red})$ when $\omega_t$ and $e$ pass through a critical level.

	Note that if the action is semifree, Hamiltonian, and $M^{S^1}$ consists of surfaces, then any reduced space $M_{\red}$ with respect to the moment map is diffeomorphic to an $S^2$-bundle over $\Sigma_g$. Indeed, if all fixed components in non-extremal levels of the moment map are of codimension~$4$, the diffeomorphism type of the reduced space does not change when passing through a critical level. Moreover, the reduced space at a critical level is a smooth manifold. (See \cite{MD} or \cite{GS}.) Therefore in order to find the diffeomorphism type of $M_{\red}$, it is enough to look at the reduced space near the minimum. The regular level near the minimum is an $S^3$-bundle over the minimum, so $M_{\red}$ is a ruled surface.

	To prove Theorem~\ref{main}, we use the fact that each fixed surface is a symplectic submanifold of the reduced space at the level in which the fixed surface lies. This will be treated in Section~\ref{Semifree $S^1$ actions} and Section~\ref{proof for nontrivial}.

	In Section~\ref{construction of example}, we construct several examples of semifree Hamiltonian $S^1$-manifolds. In~\cite{L2}, Li constructed some semifree Hamiltonian $S^1$-manifolds such that $M^{S^1}$ consists of surfaces. In her examples, the number of fixed surfaces with positive genera could be one or three. In Example~\ref{4genus}, we construct a family of examples of Hamiltonian $S^1$-manifolds whose fixed set consists of four surfaces with nonzero equal genus and any number of spheres. We also construct an example with $N$ fixed spheres for any $N \geq 4$ in Example~\ref{sphere}. Therefore, the upper bound on the number of fixed surfaces with positive genera is optimal and there is no constraint on the number of fixed spheres. The construction follows that in~\cite{L2}. To construct such Hamiltonian $S^1$-manifolds, we first construct local pieces which are obtained by the methods of Guillemin and Sternberg~\cite{GS}, namely simple cobordisms between reduced spaces, and then glue them together. A symplectic structure on the ruled surface is said to be \textit{compatible with the ruling} if all fibers are symplectic. Since two cohomologous symplectic forms on the same ruled surface which are compatible with the ruling are isotopic~\cite{LM}, we only need to check that two pieces have the same Euler class and cohomologous symplectic form on the gluing region. We discuss this in Section~\ref{construction of example} in more detail.

	In Section~\ref{background} we give a brief review about symplectic circle actions with generalized moment maps. In Section~\ref{symplectic ruled surface} we introduce results about symplectic structures on ruled surfaces due to Li and Liu~\cite{LL}, and study symplectomorphisms on ruled surfaces. In Section~\ref{proof of 1.2} we give the proof of Theorem~\ref{main-}. In Section~\ref{Semifree $S^1$ actions} and~\ref{proof for nontrivial} we give the proof of Theorem~\ref{main}.

\begin{remark}
	Here is a remark on the sign convention. In~\cite{MD}, McDuff regards the symplectic form $\omega$ on $\C^n$ as $\sum_i dx_i\wedge dy_i$, moment map~$H$ as a map satisfying $i_X \omega = dH$ for the fundamental vector field $X$ of the standard diagonal action, the Euler class of each level set as $-d\alpha$ where $\alpha$ is a connection $1$-form on a level set, and the gradient flow of $H$ as $JX$. In contrast, in \cite{Au}, \cite{L1} and \cite{L2}, they assume that the symplectic form on $\C^n$ is $\sum_i dy_i \wedge dx_i$ so that all information has opposite signs, i.e., they use the symplectic form $-\omega$ where $\omega$ is the one that McDuff used. In this paper, we use the sign setting of \cite{Au} and \cite{L1}.
\end{remark}

\subsection*{Acknowledgments}
We thank Dusa McDuff for commenting to us that examples constructed in Section~\ref{construction of example} are uniquely determined by the given data. We are indebted to the referee who read the paper with meticulous care and pointed out many inaccuracies.

\section{Background}\label{background}

	In this section we give some basic materials needed to state and prove the main theorems.

	Let $M^{2n}$ be a smooth closed connected manifold. A $2$-form $\omega$ on~$M$ is called \textit{symplectic} if it is closed and non-degenerate. Since $\omega$ is non-degenerate, $(M,\omega)$ has a natural orientation given by $\int_M \omega^{n} > 0$. We call such pair $(M, \omega)$ a \textit{symplectic manifold}. An $S^1$-action on $(M, \omega)$ is called \textit{symplectic} if it preserves the symplectic form~$\omega$. Let $X$ be the fundamental vector field of the $S^1$-action. Then the condition that the action is symplectic is equivalent to that $i_X \omega$ is closed by Cartan's formula. Furthermore, if $i_X \omega$ is exact, we say the action is \textit{Hamiltonian}. If the action is Hamiltonian, there exists a function $H : M \rightarrow \R$ satisfying $i_X \omega = dH$ which is called a \textit{moment map} of the action. It is a well-known fact that $H$ is a perfect Morse-Bott function. Note that the critical point set of~$H$ is equal to the fixed point set~$M^{S^1}$. Since $M$ is compact, the fixed point set is not empty.

	Choose an $S^1$-invariant $\omega$-compatible almost complex structure~$J$ so that $g(\cdot, \cdot) = \omega(J\cdot, \cdot)$ defines a metric on~$M$. Then $-JX$ is the gradient vector field of~$H$ with respect to the metric~$g$. By Morse Theory, for any closed regular interval $[a,b]$, the gradient flow of~$H$ gives an isotopy from $H^{-1}(a)$ to $H^{-1}(b)$. Furthermore, all critical points have even indices. In particular, every level set of~$H$ is connected.

	Now assume that the cohomology class $[\omega]$ is integral in $H^2(M,\R)$. Then $i_X \omega$ is also integral so that we can define a map $\mu : M \rightarrow S^1$ as follows (See \cite{MD} and \cite{Au}). Choose any point $x_0$ in $M$ and define $\mu(x) = \int_{x_0}^x i_X \omega$. For any paths $\sigma_1$ and $\sigma_2$ from $x_0$ to $x$, the difference of their path integrals $\int_{\sigma_1} i_X \omega - \int_{\sigma_2} i_X \omega$ is an integer, so $\mu$ is well-defined as an $S^1$-valued function $\mu : M \rightarrow \R/\Z \cong S^1$. This map is called a \textit{generalized moment map}. A generalized moment map $\mu$ satisfies many properties of the moment map. Locally $\mu$ satisfies the equation $i_X\omega = d\mu$. Note that for an $\omega$-compatible almost complex structure $J$, the infinitesimal action of the vector field $-JX$ determines an orientation of $S^1 \cong \R/\Z$, which we call \textit{the induced orientation with respect to the $S^1$-action}. Therefore we can define the index of the critical point. Note that if $\mu$ can be lifted to an $\R$-valued function, the lifted map is the moment map of the given action. Hence the action is Hamiltonian.

\begin{remark}\label{RR}\cite{R}
	Given a Hamiltonian $S^1$-action on $(M,\omega)$, consider another symplectic form $\omega'$ on $M$ such that the given $S^1$-action is symplectic with respect to $\omega'$. Then the action is also Hamiltonian with respect to $\omega'$. This follows from Proposition \ref{R} below, i.e., if we let $C_0$ be a fixed component which is sent to the minimum by the moment map with respect to $\omega$, then any loop $\sigma$ in $M$ is homotopic to some loop in $C_0$. If we change an $S^1$-invariant symplectic structure by $\omega'$ and $\mu$ is a generalized moment map with respect to $\omega'$, then $\mu_*(\sigma)$ should be zero in $\pi_1(S^1)$ so that $\mu$ can be lifted to an $\R$-valued moment map. Thus, if we want to prove a symplectic $S^1$-action on $(M,\omega)$ to be Hamiltonian, it is enough to prove the claim with the assumption that $[\omega]$ is integral.
\end{remark}

\begin{proposition}\label{R}\cite{R}
	Let $M$ be a compact connected Riemannian manifold and $f$ a Morse-Bott function with no critical manifold of index~$1$. Then, there is only one connected critical manifold $C_0$ of index $0$ and $\pi_1(M/C_0) = 0$. In fact, for any loop $\sigma \in \pi_1(M)$, there is a homotopy from $\sigma$ to some loop in $C_0$ relative to the points in $\sigma \bigcap C_0$ along the negative gradient flow of~$f$.
\end{proposition}

\begin{remark}\label{morse}
	By a Morse theoretic argument, if $M$ is a differentiable manifold of dimension $n$ and if the Morse function $f : M \rightarrow \R$ has no critical point of index $1$ or $n-1$, then the number of connected components of $f^{-1}(t)$ is constant for all $t \in \textrm{Im}f \subset \R$. Moreover, the number of connected components of a level set is equal to the number of connected components of $M$. In particular, if $M$ is connected, then every level set is connected. (See \cite{Au} p.112)
\end{remark}

    The following proposition seems to be known to people, but it is difficult to find proper references for the proof. We thereby provide the proof here.

\begin{proposition}\label{connect}

	Let $(M,\omega)$ be a connected closed $2n$-dimensional symplectic $S^1$-manifold with integral symplectic form, i.e., $[\omega] \in H^2(M,\Z)$. Then, there is a generalized moment map $\mu : M \rightarrow S^1$ such that $\mu^{-1}(t)$ is connected or empty for each $t \in S^1$. In particular, if the action is non-Hamiltonian, then we have a generalized moment map $\mu$ such that $\mu^{-1}(t)$ is non-empty and connected for all $t \in S^1$.

\end{proposition}

	\begin{proof}
	If the action is Hamiltonian, then $\mu$ can be lifted to an $\R$-valued moment map $\widetilde{\mu}$. Therefore by Remark~\ref{morse}, $\widetilde{\mu}^{-1}(t)$ is connected or empty for each $t \in \R$. Let $m > \max\ \widetilde{\mu} - \min\ \widetilde{\mu}$ be an integer. For a covering map $p : \R \rightarrow S^1$ given by $\R \rightarrow \R / m\Z$, we get a generalized moment map $\mu := p \circ \widetilde{\mu}$ which satisfies the condition of Proposition~\ref{connect}.

	Now assume that the action is non-Hamiltonian. Let $\mu : M \rightarrow S^1$ be a generalized moment map as defined before Remark~\ref{RR}. Since all indices of the critical points are even, there is no critical point of index $1$ or $2n-1$ in $M$. This means that the number~$k$ of connected components of $\mu^{-1}(t)$ is constant for all $t \in S^1$. We may assume $k>1$. Note that $\mu$ induces a group homomorphism $\mu_* : \pi_1(M) \rightarrow \pi_1(S^1) \cong \Z$.

	We claim that $\text{Im}\mu_* = k  \Z \subset \Z$ so that $\mu$ can be lifted to the $k$-fold covering of $S^1$. Fix a regular value $t_0 \in S^1 \cong \R / \Z$ and let $L_1,L_2, .. , L_k$ be the components of $\mu^{-1}(t_0)$. For a sufficiently small $\epsilon$, consider a Hamiltonian $S^1$-manifold $\mu^{-1}(S^1-(t_0-\epsilon,t_0+\epsilon))$ with boundary, regarding $S^1-(t_0-\epsilon,t_0+\epsilon)$ as a closed interval in $\R$ with moment map $\mu$. Note that $\mu^{-1}((t_0-\epsilon,t_0+\epsilon)) \cong \coprod L_i \times (t_0-\epsilon, t_0 + \epsilon)$. Then $\mu^{-1}(t_0-\epsilon) \cong \coprod L_i \times (t_0 - \epsilon)$ is the maximum level set of $\mu$ and $\mu^{-1}(t_0+\epsilon) \cong \coprod L_i \times (t_0 + \epsilon)$ is the minimum level set of $\mu$ in $\mu^{-1}(S^1-(t_0-\epsilon,t_0+\epsilon))$.

	By Remark~\ref{morse}, $\mu^{-1}(t)$ has $k$ components for all $t \in S^1-(t_0-\epsilon,t_0+\epsilon)$. Let $M_i$ be the connected component of $\mu^{-1}(S^1-(t_0-\epsilon,t_0+\epsilon))$ whose minimum level set is $L_i \times (t_0 + \epsilon)$. Then M is decomposed by disjoint pieces $\coprod L_i \times (t_0-\epsilon, t_0 + \epsilon)$ and $\coprod M_i$. If the maximum level set of $M_i$ is $L_i \times (t_0 - \epsilon)$, then it implies that there is no path in $M$ from $L_i \times (t_0-\epsilon,t_0+\epsilon)$ to $L_j \times (t_0-\epsilon,t_0+\epsilon)$ for any $j \neq i$, which contradicts to the assumption that $M$ is connected. So, the maximum level set is $L_{\sigma(i)} \times (t_0 - \epsilon)$ for some $\sigma(i) \neq i$. By a similar argument, we deduce that $\sigma : [k] \rightarrow [k]$ is a permutation cycle of length~$k$.

	Choose any loop $\tau : [0,1] \rightarrow M$ with $\tau(0) = \tau(1) = x$ such that $\mu_*([\tau]) \neq 0$ in $\pi_1(S^1)$. Without loss of generality, we may assume that $x \in L_1$. If $\tau$ goes around along the orientation preserving direction, then $\tau$ passes through each $L_i$ for all $i=1, ... , k$ and the numbers of intersection points with $L_i$ are all equal by the previous argument. So, $\mu_*([\tau]) \equiv 0 \,\,(\mathrm{mod} \,\, k)$. We need the following lemma to finish the proof of the claim.

\begin{lemma}
	Let $(M,\omega)$ be a connected closed symplectic $S^1$-manifold.
	\begin{itemize}
		\item[(i)] Suppose the action is Hamiltonian with moment map $\mu : M \rightarrow \R$. Let $r$ be a regular value of~$\mu$ in~$\R$. For any points $x,y$ which lie in $\mu^{-1}(r)$ and for any path $\alpha$ connecting $x$ and~$y$ in~$M$, $\alpha$ is homotopic to a path in $\mu^{-1}(r)$ relative to~$x$ and~$y$.
		\item[(ii)] With the same assumption, for any points $x$ in the minimum level set and $y$ in the maximum level set of~$\mu$,  any path from~$x$ to~$y$ is homotopic to some path $\beta : [0,1] \rightarrow M$ relative to $x$ and $y$ such that $\mu \circ \beta$ is non-decreasing.
		\item[(iii)] Suppose the action is non-Hamiltonian with a generalized moment map $\mu : M \rightarrow S^1$. Then any loop in M is homotopic to some loop $\gamma : [0,1] \rightarrow M$ relative to the base point such that $\mu \circ \gamma$ is non-decreasing or non-increasing with respect to the induced orientation on $S^1$.
	\end{itemize}
\end{lemma}

	\begin{proof}
	(i) For a regular value $r \in \R$ with $x$,$y$ $\in \mu^{-1}(r)$, let $\alpha : [0,1] \rightarrow M$ be a path from $x$ to $y$. Since $[0,1]$ is compact, we can find a partition $0 = a_1 < a_2 < \dots < a_k = 1$ such that $\mu \circ \alpha|_{[a_i,a_{i+1}]} \geq r$ or $\leq r$ and $\mu(\alpha(a_i)) = r$ for $i = 1,2, \dots , k-1$. Assume that $\mu \circ \alpha|_{[a_1,a_2]} \geq r$. Since $\mu^{-1}(r)$ is connected, there is a path $\delta$ from $\alpha(a_2)$ to $\alpha(a_1)$ on $\mu^{-1}(r)$. Then $\delta * \alpha|_{[a_1,a_2]}$ is a loop at $\alpha(a_1)$ such that $\mu (\delta * \alpha|_{[a_1,a_2]}) \geq r$ where $*$ is the product of paths. By Proposition \ref{R}, this loop $\delta * \alpha|_{[a_1,a_2]}$ is homotopic to some loop in $\mu^{-1}(r)$ by considering the manifold $\mu^{-1}([r,\infty))$ and we can choose such homotopy fixing $\delta$. It gives a homotopy from the path $\alpha|_{[a_1,a_2]}$ to some path on $\mu^{-1}(r)$ such that $\alpha(a_1)$ and $\alpha(a_2)$ are fixed. Similarly, we can find the homotopy on each $(a_i, a_{i+1})$ for all $i$.

	(ii) Let $\alpha$ be a path from $x$ to $y$. Then we may choose a partition $0 = a_1 < a_2 <  ... < a_k = 1$ such that $\mu \circ \alpha$ is non-decreasing on $(a_i, a_{i+1})$, or $\alpha(a_i) = \alpha(a_{i+1})$. In the latter case, we apply (i) to get a non-decreasing path.

	(iii) Consider the following diagram
\[\xymatrix{\widetilde{M} \ar[r]^{\widetilde{\mu}} \ar[d]^{\pi}& \R \ar[d]^{p}\\ M \ar[r]_{\mu}& S^1}\]
where $p$ is a universal cover of~$S^1$, and ${\widetilde{M}}$ is the pull-back manifold for~$\mu$. Then ${\widetilde{M}}$ is a $\Z$-fold covering and a non-compact manifold.

		For any loop~$\alpha$ in~$M$, there is a unique path lifting ${\widetilde{\alpha}}$ on ${\widetilde{M}}$ for a given starting point of ${\widetilde{\alpha}}$. Moreover ${\widetilde{\mu}}$ is a Morse function with critical points of even indices on ${\widetilde{M}}$. By (ii), we can get a path ${\widetilde{\gamma}}$ which is non-decreasing along the gradient flow of ${\widetilde{\mu}}$. It is easy to see that $\gamma = \pi \circ \widetilde{\gamma}$ is the loop we want.
\end{proof}

	Now, back to the proof of Proposition~\ref{connect}. If we choose any generalized moment map $\mu$, we get a lifting ${\widetilde{\mu}} : M \rightarrow S^1$ whose level sets are all connected. If we regard the $k$-fold covering~$p$ of $S^1$ as a map $\R / k\Z \rightarrow \R / \Z$, then $p$ is a local isometry and we have $ \iota_X \omega = d\mu = dp \circ d{\widetilde{\mu}} = d{\widetilde{\mu}}$ so that ${\widetilde{\mu}}$ satisfies the Hamiltonian equation locally. Hence $\widetilde{\mu}$ is a desired generalized moment map.
\end{proof}

	From now on, we assume that level sets of the generalized moment map are connected. Assume that the symplectic $S^1$ action on $(M, \omega)$ is \textit{semifree}, i.e., it is free on $M \backslash M^{S^1}$. Let $\mu : M \rightarrow S^1$ be a generalized moment map. For a regular value $t \in S^1$ of~$\mu$, the inverse image $\mu^{-1}(t)$ is a $(2n-1)$-dimensional free $S^1$-manifold so that $\mu^{-1}(t) \rightarrow \mu^{-1}(t)/S^1$ is a principal $S^1$-bundle over $\mu^{-1}(t)/S^1$. Since the action is free on regular levels,  $\mu^{-1}(t)/S^1$ is a connected smooth manifold with the induced symplectic structure. It is called a \textit{symplectic reduction} or \textit{reduced space} of~$M$ at~$t$, and denoted by $(M_t, \omega_t)$ where $\omega_t$ is the reduced symplectic form on~$M_t$.

	Assume that all fixed components in~$M$ are of codimension~$4$. Then all reduced spaces at any level~$t$ have the same diffeomorphism type (even if $t$ is a critical value, see \cite{GS}) and we denote their representative by~$M_{\red}$. In fact, all reduced spaces are isotopic to each other along the Morse flow. So we may consider the set of all pairs $(M_t, \omega_t)$ as an $S^1$-parametrized family of symplectic manifolds $(M_{\red}, \omega_t)$. Moreover, each fixed component $X_i$ is symplectically embedded in the reduced space at the critical level containing~$X_i$. When $t$ varies in a regular interval $(a,b) \in S^1$, by the Duistermaat-Heckman theorem, we have
\begin{equation}\label{DH}
	[\omega_t] = [\omega_r] - e(t-r) \in H^2(M_{\red},\R)
\end{equation}
where $r,t \in (a,b)$, $e \in H^2(M_{\red}, \Z)$ is the Euler class of the $S^1$-bundle $\mu^{-1}(t) \rightarrow M_{\red}$ (See~\cite{Au} or~\cite{DH}). Let $s \in S^1$ be a critical value of $\mu$, and let $X_1, ... , X_k$ be the fixed components of codimension 4 contained in $\mu^{-1}(s)$. Then the change of the Euler class for the principal $S^1$-bundle over $M_{\red}$ is given by
\begin{equation}\label{class}
	e_{s+} = e_{s-} + \sum_{X_i \in \mu^{-1}(s)}{D(X_i)}.
\end{equation}
Here, $e_{s-}$, respectively $e_{s+}$, denotes the Euler class at levels just below, respectively above, the critical level $s$ and $D(X_i) \in H^2(M_{\red},\Z)$ denotes the Poincar\'e dual of $[X_i] \in H_{2n-4}(M_{\red},\Z)$. See \cite{GS} Theorem~13.2 for more details.

	Suppose that $(M,\omega)$ is a non-Hamiltonian $S^1$-manifold with a generalized moment map $\mu : M \rightarrow S^1$. Fix a regular value $r$. Then $(M,\omega)$ can be reconstructed as follows. Consider the pull-back manifold
\[\xymatrix{(\widetilde{M}, \widetilde{\omega}) \ar[r]^{\widetilde{\mu}} \ar[d]^{\pi}& \R \ar[d]^{p}\\ (M,\omega) \ar[r]_{\mu}& S^1}\]
where $\widetilde{\omega} = \pi^* \omega$. Let $[t_{\min}, t_{\max}] \subset \R$ be a closed interval such that $p(t_{\min}) = p(t_{\max}) = r$ and $p([t_{\min}, t_{\max}])$ is a generator of $\pi_1(S^1)$.  Let $\widetilde{M'}$ be the preimage of $[t_{\min}, t_{\max}]$ by $\widetilde{\mu}$. Then $(\widetilde{M'},\widetilde{\omega}|_{M'})$ is a Hamiltonian $S^1$-manifold with two boundary components $M_+ = \widetilde{\mu}^{-1}(t_{\max})$ and $M_- = \widetilde{\mu}^{-1}(t_{\min})$. Note that $\pi|_{M_+}$, respectively $\pi|_{M_-}$, is an $S^1$-equivariant diffeomorphism from $M_+$, respectively $M_-$, to $\mu^{-1}(r)$ and the induced map $\pi|_{M_+} / S^1$, respectively $\pi|_{M_-} / S^1$, is a symplectomorphism from $M_+ / S^1$, respectively $M_- / S^1$, to $(\mu^{-1}(r)/S^1, \omega_r)$ where $\omega_r$ is the reduced symplectic form at level $r$. Now, define $M' = \widetilde{M'} / \phi$ to be the quotient space where $\phi : M_+ \rightarrow M_-$ is defined by $\phi = \pi|_{M_-}^{-1} \circ \pi|_{M_+}$.

\begin{proposition}\label{cor}
	The above map $\pi|_{\widetilde{M'}} : \widetilde{M'} \rightarrow M$ induces an $S^1$-equivariant symplectomorphism $\pi|_{\widetilde{M'}} / \phi :  M' \rightarrow M$.
\end{proposition}

	\begin{proof}
	To show that $\pi|_{\widetilde{M'}} / \phi$ is well-defined, we only need to check that $\pi|_{\widetilde{M'}}(x) = \pi|_{\widetilde{M'}} (\phi(x))$ for any $x \in M_+$. For $x \in M_+$, $$\pi|_{\widetilde{M'}} (\phi(x)) = \pi|_{\widetilde{M'}} \circ \pi|_{M_-}^{-1} \circ \pi|_{M_+} (x) = $$ $$\pi|_{M_-} \circ \pi|_{M_-}^{-1} \circ \pi|_{M_+} (x) = \pi|_{M_+} (x) = \pi|_{\widetilde{M'}}(x).$$ The rest of the proof is straightforward.
	\end{proof}

	So we can regard the closed non-Hamiltonian $S^1$-manifold $(M,\omega)$ as a pair $(M',\phi)$ where $M'$ is a Hamiltonian $S^1$-manifold with two boundary components $M_+, M_-$ which are the maximum and the minimum of the moment map respectively, and $\phi$ is an $S^1$-equivariant diffeomorphism which induces a symplectomorphism on their reduced spaces.

\section{Symplectic Ruled Surfaces}\label{symplectic ruled surface}

	In this section, let $(M,\omega)$ be a $6$-dimensional closed symplectic $S^1$-manifold with a generalized moment map such that $M_{\red}$ is a ruled surface. A \textit{ruled surface} is an $S^2$-bundle over a compact Riemann surface~$\Sigma_g$ of genus~$g$. Note that the structure group of an oriented $S^2$-bundle over $\Sigma_g$ is $SO(3)$ whose fundamental group is $\Z_2$. So we can easily see that there are only two diffeomorphism types in the set of oriented $S^2$-bundles over $\Sigma_g$: the trivial bundle $\Sigma_g \times S^2$ and the non-trivial one denoted by $E_{\Sigma_g}$. A ruled surface is called \textit{rational} if the base manifold is a sphere, and \textit{irrational} otherwise. A ruled surface with a symplectic structure is called a symplectic ruled surface.

	A symplectic form on a ruled surface is said to be \textit{compatible with the ruling} (or with the fiber) if its restriction to the fiber is non-degenerate. This means that the fiber of a ruled surface is a symplectic submanifold. The following theorem is due to Li and Liu. Using this theorem, we may assume that the fiber of a ruled surface is symplectic.

\begin{theorem}\label{LL}(\cite{LL}, \cite{McS} p.448)
	Let $\pi : M \rightarrow \Sigma_g$ be a smooth $S^2$-bundle over a compact Riemann surface $\Sigma_g$. For any symplectic form $\omega$ on $M$, $(M,\omega)$ is symplectomorphic to $(M,\omega')$ for some symplectic form $\omega'$ which is compatible with the given ruling $\pi$. Moreover, we can assume that this symplectomorphism acts trivially on homology.
\end{theorem}

	Let $(M_{\red}, \omega_{\red})$ be a symplectic ruled surface. We choose a basis of $H^2(M_{\red},\Z)$ as follows. When $M_{\red}$ is trivial, let $u$ be the dual of the class represented by a symplectic fiber $pt \times S^2$, and let $v$ be the dual of the class represented by a trivial section $\Sigma_g \times pt$ such that $\int_{M_{\red}} uv = 1$. For the nontrivial case, let $u$ be the dual of the class represented by a symplectic fiber and $v$ be the dual of the class represented by a section with self-intersection $-1$ with $\int_{M_{\red}} uv = 1$. Note that $\int_{M_{\red}} u^2 = \int_{M_{\red}} v^2 = 0$ and $\int_{M_{\red}} uv = 1$ in the trivial case,  and $\int_{M_{\red}} u^2 = 0$, $\int_{M_{\red}} v^2 = -1$ and $\int_{M_{\red}} uv = 1$ in the non-trivial case.

	We compute the first Chern class $c_1(M_{\red})$ as follows. Let $c_1(M_{\red}) = xu + yv$ for some $x,y \in \Z$. Let $F \cong S^2$ be the fiber representing the dual of $u$, and $B \cong \Sigma_g$ be the section of $M_{\red}$ representing the dual of $v$.
	
	In the trivial case $M_{\red} \cong \Sigma_g \times S^2$, the normal bundles $\nu(F)$ and $\nu(B)$ in $M$ are trivial.
\[\int_F c_1(M_{\red})|_F = \int_{M_{\red}} c_1(M_{\red}) \cdot u = \int_{M_{\red}}(xu + yv) \cdot u = y.\]
On the other hand, $c_1(M_{\red})|_F = c_1(TF \oplus \nu(F)) = c_1(TF) = c_1(S^2)$. Therefore $y = \int_F c_1(M_{\red})|_F = 2$. Similarly
\[\int_B c_1(M_{\red})|_B = \int_{M_{\red}} c_1(M_{\red}) \cdot v = \int_{M_{\red}}(xu + yv) \cdot v = x.\]
On the other hand, $c_1(M_{\red})|_B = c_1(TB \oplus \nu(B)) = c_1(TB) = c_1(\Sigma_g)$. Therefore $x = \int_B c_1(M_{\red})|_B = 2-2g$. Hence
\begin{equation}\label{chern}
	c_1(M_{\red}) = (2-2g)u + 2v.
\end{equation}
	In the non-trivial case $M_{\red} \cong E_{\Sigma_g}$, the normal bundle $\nu(F)$ is trivial, but the normal bundle $\nu(B)$ is the complex line bundle whose first Chern number is equal to $-1$. A similar computation shows that $$2 = \int_F c_1(M_{\red})|_F =  \int_{M_{\red}} c_1(M_{\red}) \cdot u = y.$$ Since $\int_B c_1(M_{\red})|_B = \int_B c_1(TB \oplus \nu(B)) = 1-2g$, $$ 1-2g = \int_B c_1(M_{\red})|_B =  \int_{M_{\red}} c_1(M_{\red}) \cdot v = x-y.$$ Therefore $x = 3-2g$ and
\begin{equation}\label{chern2}
	c_1(M_{\red}) = (3-2g)u + 2v.
\end{equation}

	Now, assume that $(M,\omega)$ is a $6$-dimensional closed non-Hamiltonian symplectic semifree $S^1$-manifold with a generalized moment map $\mu$. Assume $M^{S^1} = \coprod_{i=1}^k X_i$ where $X_i$'s are connected surfaces for all $i$, and assume that the reduced space $M_{\red}$ is a ruled surface. By Proposition \ref{cor}, identify $(M,\omega)$ with $(M',\phi)$. Here, $M'$ is a Hamiltonian $S^1$-manifold with two boundary components $M_+$ and $M_-$ with $(M')^{S^1} = \coprod_{i=1}^k X_i$, and $\phi : M_+ \rightarrow M_-$ is an $S^1$-equivariant diffeomorphism described as in Proposition \ref{cor}. Because $M_{\red}$ is a ruled surface, $M_+$ and $M_-$ are $S^1$-bundles over the ruled surface $M_{\red}$.

	Let $e_-$, respectively $e_+$, be the Euler class of the $S^1$-bundle $M_-$, respectively $M_+$. The $S^1$-equivariant diffeomorphism $\phi : M_+ \rightarrow M_-$ induces a symplectomorphism $M_+/S^1 \rightarrow M_-/S^1$, still denoted by $\phi$. Then $\phi^* e_- = e_+$ and the equation~\eqref{class} implies
\begin{equation}\label{gluing}
	\phi^*e_- = e_- + \sum_{i}{D(X_i)}
\end{equation}
In particular, if $\phi$ acts on $H^*(M_{\red})$ trivially, then $\sum_{i}{D(X_i)} = 0$.

\begin{proposition}\label{attach}
	Let $M_{\red}$ be a ruled surface with basis $\{u, v\}$ of $H^2(M_{\red}, \Z)$ as above. Let $\omega_1$ and $\omega_2$ be two symplectic forms on $M_{\red}$ which are both compatible with the ruling. Let $\psi : (M_{\red}, \omega_1) \rightarrow (M_{\red}, \omega_2)$ be a symplectomorphism. Then $\psi$ acts trivially on $H^2(M_{\red}, \Z)$ if $M_{\red} \ncong S^2 \times S^2$. If $M_{\red} \cong S^2 \times S^2$, then $\psi$ acts on $H^2(M_{\red}, \Z)$ either trivially, or $\psi^*u = v$, $\psi^*v = u$.
\end{proposition}
\begin{proof}
	First assume $M_{\red} \cong E_{\Sigma_g}$ is a non-trivial $S^2$-bundle. Then $u^2 = 0$, $v^2 = -1$, $uv = 1$. Let $\psi^* u = au + bv$ and $\psi^* v = cu + dv$. Since $\psi$ is a symplectomorphism, it preserves the intersection form. Hence $\psi^* u^2 = (au+bv)^2 = 0$, $\psi^* v^2 = (cu+dv)^2 = -1$, and $\psi^* u \cdot \psi^* v = (au+bv)(cu+dv) = 1$. The possible integral solutions are $(a,b,c,d) = (1,0,0,1), (-1,-2,0,1), (-1,0,0,-1)$, and $(1,2,0,-1)$. Note that for any $2$-dimensional submanifold $Z \subset M_{\red}$, genus($Z$) = genus($\psi(Z)$).

	Assume $(a,b,c,d) = (-1,-2,0,1)$, i.e., $\psi^* u = -u - 2v$ and $\psi^* v = v$. Let $Z$ be a symplectic $2$-sphere which represents $u$. Apply the adjunction formula to the 2-sphere $Z$ so that
\begin{equation}\label{adjunc}
	\langle \psi^* u \cdot \psi^* u, [M_{\red}] \rangle - \langle c_1(M_{\red}) \cdot \psi^* u , [M_{\red}] \rangle = 2g(\psi(Z)) - 2 = -2.
\end{equation}
Here, $\textrm{LHS} = (-u-2v)^2 - ((3-2g)u + 2v)(-u -2v) = 4-4g$, which is impossible since $g$ is an integer. Therefore $(a,b,c,d) \neq (-1,-2,0,1)$. Similarly $(a,b,c,d) = (1,2,0,1)$ and $(a,b,c,d) = (-1,0,0,-1)$ are not possible. This proves the result when $M_{\red} \cong E_{\Sigma_g}$.

	Now assume $M_{\red} \cong \Sigma_g \times S^2$. Then $u^2 = 0$, $v^2 = 0$, and $uv = 1$. Hence $\psi^* u^2 = (au+bv)^2 = 0$, $\psi^* v^2 = (cu+dv)^2 = 0$, and $\psi^* u \cdot \psi^* v = (au+bv)(cu+dv) = 1$. Equivalently, $ab = 0$, $cd = 0$, and $ad + bc = 1$. The only possible integral solutions are $(a,b,c,d) = (0,\pm 1, \pm 1, 0), (\pm 1, 0, 0, \pm 1)$, i.e. $\psi ^* u = \pm v, \psi^* v = \pm u$, or $\psi^* u = \pm u, \psi^* v = \pm v$.

	If $(a,b,c,d)$ is $(0,-1,-1,0)$ or $(-1,0,0,-1)$, then the fiber cannot be symplectic with respect to $\omega_2$. So, $(a,b,c,d) \neq (0,-1,-1,0), (-1,0,0,-1)$. Let $Z$ be a symplectic surface of genus $g$ which represents $v$. If $(a,b,c,d) = (0,1,1,0)$, then $\psi ^* u = v, \psi^* v =  u$. By applying the adjunction formula to the surface $\psi(Z)$,
\[\langle \psi^* v \cdot \psi^* v, [M_{\red}] \rangle - \langle c_1(M_{\red}) \cdot \psi^* v , [M_{\red}] \rangle = 2g(\psi(Z)) - 2 = 2g - 2.\]
$\textrm{LHS} = u^2 - ((2-2g)u + 2v)u = -2$. Therefore we get $g=0$. Consequently $\psi^* = id$ if $g \geq 1$, and $\psi^* = id$ or $\psi^*u=v$ and $\psi^*v = u$ if $g=0$.
\end{proof}

	If $\phi$ acts non-trivially on $H^2(M_{\red})$, we say that $M$ is \textit{twisted}. If $M_{\red}$ is a ruled surface not diffeomorphic to $S^2 \times S^2$, we can parametrize the cohomology classes of the reduced symplectic forms $[\omega_t] = c_tu + d_tv$ by $t \in S^1$ and $c_t, d_t$ are continuous on $S^1$.

	If $M_{\red} \cong S^2 \times S^2$, we can also parametrize the cohomology classes of the reduced symplectic forms $[\omega_t] = c_tu + d_tv$ by $t \in S^1 $. But in this case, $c_t$ and~$d_t$ may not be continuous at one point when $M$ is twisted.

	For each fixed component $X_i$, if we write $D(X_i) = a_iu + b_iv$ for $a_i, b_i \in \Z$, then
\begin{equation}\label{zerosum}
	\sum_i a_i = \sum_i b_i = 0
\end{equation}
if $M_{\red} \ncong S^2 \times S^2$ or $M_{\red} \cong S^2 \times S^2$ but $M$ is not twisted.

	When $M_{\red} \cong S^2 \times S^2$ and $M$ is twisted, let $e_- = au + bv$. By~\eqref{gluing} and Proposition~\ref{attach}, we have
\begin{equation}\label{plus}
	bu + av = au + bv + \sum_i D(X_i) = (a+\sum_i a_i)u + (b+\sum_i b_i)v.
\end{equation}
Hence from \eqref{zerosum} and \eqref{plus}, we always have the following equation
\begin{equation}\label{zerosum2}
	\sum_i (a_i + b_i) = 0.
\end{equation}

\section{Proof of Theorem \ref{main-}}\label{proof of 1.2}

We start with the following theorem due to \cite{Lin}.
\begin{theorem}\label{Lin}\cite{Lin}
	Assume that the action of a torus $T$ on a connected closed symplectic manifold $M$ is an effective Hamiltonian action of complexity two with moment map $\Phi : M \rightarrow \mathfrak{t}^{*}$. And assume that for any regular value $\xi \in \mathfrak{t}^{*}$ of $\Phi$, the symplectic reduced space $M_{\xi} = \Phi^{-1}(\xi)/T$ has $b_2^{+} = 1$. Then the Duistermaat-Heckman function $f$ is log-concave.
\end{theorem}

	Recall that the Duistermaat-Heckman function $f : \mathfrak{t}^{*} \rightarrow \R$ is defined to be the symplectic volume $f(\xi) = \int_{M_{\xi}} \omega_{\xi}^{\dim M_{\xi}/2}$, where $M_{\xi}$ is the reduced space at level~$\xi$ and $\omega_{\xi}$ is the corresponding reduced symplectic form. Duistermaat and Heckman proved that $f$ is a piecewise polynomial function and that it is a polynomial on any open region $U \in \mathfrak{t}^{*}$ consisting of regular values of $\Phi$. We sketch the proof of Theorem~\ref{Lin} in the case when $T = S^1$.

	Let $I$ be an open interval consisting of regular values of~$\Phi$ and let $e$ be the Euler class of the $S^1$-fibration over $M_{\xi}$ for $\xi \in I$. Note that $e$ does not depend on the choice of $\xi \in I$. Choose a fixed $\xi_0 \in I$. By \eqref{DH}, $[\omega_{\xi}] = [\omega_{\xi_0}] - e(\xi - \xi_0)$. By assumption, the reduced space $M_{\xi}$ is of dimension~$4$, so the function~$f$ is given by $f(\xi) = \int_{M_{\xi}} ([\omega_{\xi_0}] - e(\xi - \xi_0))^2$. The condition $b_2^{+} = 1$ implies that there is a basis $\{u, v_1, ... , v_k\}$ on $H^2(M_{\xi}, \R)$ such that $\int_{M_{\xi}}u^2 = 1$, $\int_{M_{\xi}}v_i^2 = -1$, and $\int_{M_{\xi}}uv_i = \int_{M_{\xi}}v_iv_j = 0$ for all $i \neq j$. After expressing  $[\omega_{\xi_0}]$ and $e$ in this basis, we can prove that $f''f - f'^2 \leq 0$, so $\log f$ is concave on $I$. On the other hand, an argument of Graham~\cite[Section~3]{Gr} proves that log-concavity holds near a neighborhood of the critical values. See Theorem~\ref{3.2}. This proves that $\log f$ is concave on~$\mathfrak{t}^*$. For more details, see \cite{Lin}.

    Let $(M,\omega)$ be a 6-dimensional closed non-Hamiltonian symplectic $S^1$-manifold with a generalized moment map $\Phi : M \rightarrow S^1$. If the fixed point set $M^{S^1}$ is non-empty, Theorem \ref{main-} states that the reduced space $M_{\xi} = \Phi^{-1}(\xi) / S^1$ should have $b_2^{+} \geq 2$ for any choice of the generalized moment map $\Phi$ and the regular value $\xi$. Now we start the proof of Theorem~\ref{main-} with the following theorem due to Guillemin, Lerman, and Sternberg.

\begin{theorem}\cite[Section~5]{GLS}\label{3.2}
	Let $c$ be a critical value of the moment map~$\Phi : M \rightarrow \R$. The Duistermaat-Heckman function $f(t)$ at $t>c$ near~$c$ is given by
\[f(t) - f_-(c) = \sum_{i} \frac{\mathrm{vol}(C_i)}{(d_i-1)! \prod_j w_{ij}} (t-c)^{d_i-1} + O\left((t-c)^{d_i}\right),\]
where $f_-(c)$ denotes the limit from the left, the sum is taken over all components $C_i$ of $M^{S^1} \cap \Phi^{-1}(c)$, $d_i$ is half the real codimension of $C_i$ in~$M$, and $w_{ij}$ are the weights of the $S^1$-representation on the normal bundle of $C_i$. The symplectic volume $\mathrm{vol}(C_i)$ is defined to be $1$ if $C_i$ is a point.
\end{theorem}

\begin{proof}[Proof of Theorem \ref{main-}]
	The only if part was proved in Section~5 in~\cite{Lin}. The idea is that the reduced space near the minimum is either a weighted projective space $\CP^2$ or a weighted $\CP^1$-bundle over the minimum fixed component, and $b_2^+(M_{\xi})$ does not change when $\xi$ passes through a critical level.

	Conversely, assume that the action is non-Hamiltonian so that there is a generalized moment map $\Phi : M \rightarrow S^1$ which is surjective. Also, assume that $b_2^+(M_{\xi}) = 1$ for a regular value $\xi$ of $\Phi$. Since log-concavity of a real-valued function is a local property, we may apply Theorem \ref{Lin} to $\Phi$ so that the Duistermaat-Heckman function $f : S^1 \rightarrow \R$ with respect to $\Phi$ is log-concave on~$S^1$. Since every periodic concave function is constant, $\log{f}$ is constant so that $f$ is a constant function on $S^1$. Hence it is enough to show that $f$ can never be a constant function if the fixed point set $M^{S^1}$ is non-empty.

Assume that there is some fixed surfaces $Z_1, \cdots, Z_k$ in $M^{S^1}$ with $\Phi(Z_i) = c$. By Theorem \ref{3.2}, the Duistermaat-Heckman function on $(c, c + \epsilon)$ is a polynomial function which is of the form $$ f(t) = \sum_{i} \frac{\mathrm{vol}(Z_i)}{\prod_j w_{ij}} (t-c) + O\left((t-c)^{2}\right). $$
Since the action is assumed to be non-Hamiltonian, every index of $Z_i$ is two so that the coefficient of $t$ is negative. In particular, the Duistermaat-Heckman function $f(t)$ is not constant on $(c, c + \epsilon)$. Hence it is a contradiction.

Now assume that all fixed components are isolated points. Since the action is assumed to be non-Hamiltonian, each fixed point has index $2$ or $4$. Let $c \in S^1$ be a critical value of $\Phi$ and let $\{z_1, \cdots, z_k\}$ be the set of all fixed points which lie on $\Phi^{-1}(c)$. By Theorem \ref{3.2}, the Duistermaat-Heckman function on $(c, c + \epsilon)$ is a polynomial function which is of the form $$ f(t) = \sum_{i} \frac{1}{\prod_j w_{ij}} (t-c)^2 + O\left((t-c)^{3}\right).$$
Note that $\prod_j w_{ij}$ is positive (negative, respectively) if $z_i$ is a fixed point of index four (two, respectively). So it seems to be possible that the coefficient of $t$ would be zero if the set $\{z_1, \cdots, z_k\}$ contains index-two points and index-four points simultaneously. To resolve this situation, set $z = z_1$ and let's think of a new symplectic $S^1$-manifold $(\widetilde{M}, \widetilde{\omega})$, which is a sufficiently small $S^1$-equivariant symplectic blow up at $z$ with an exceptional divisor $E$. Note that the symplectic blow up keeps unchanged everything away from the blown up point. Therefore, $\Phi$ induces a generalized moment map $\widetilde{\Phi} : \widetilde{M} \rightarrow S^1$, which coincides with $\Phi$ when we restricts $\widetilde{\Phi}$ to the complement of $E$. In particular, the Duistermaat-Heckman function $\widetilde{f}$ for $\widetilde{\Phi}$ is log-concave and satisfies Theorem~\ref{3.2} so that $\widetilde{f}$ should be constant. The fixed point set lying on the exceptional divisor $E$ consists of either three points or one point and one sphere. In any case, $E$ contains an isolated fixed point $p$ with $\widetilde{\Phi}(p) \neq c$. Since the value $\widetilde{\Phi}(p)$ depends on the size of $E$, we can choose the size of $E$ small enough so that the level set $\widetilde{\Phi}^{-1}(\widetilde{\Phi}(p))$ does not contain any other fixed point except for $p$. If we apply Theorem \ref{3.2} to $p$, then the coefficient of $t^2$ on $(\Phi(p), \Phi(p) + \epsilon)$ is non-zero for any small $\epsilon$, which contradicts that $\widetilde{f}$ is a constant. It finishes the proof.
\end{proof}

\section{Proof of Theorem \ref{main} : The trivial fibration cases}\label{Semifree $S^1$ actions}

    We first prove the following proposition which is the first statement of Theorem~\ref{main} for the trivial fibration case using a different approach from Theorem~\ref{main-}.

\begin{proposition}\label{thm1}
	Let $(M,\omega)$ be a $6$-dimensional closed symplectic semifree $S^1$-manifold with non-empty fixed point set $M^{S^1}$ consisting of surfaces. Let $\mu : M \rightarrow S^1$ be a generalized moment map. If $M_{\red}$ is diffeomorphic to $\Sigma_g \times S^2$, then the action is Hamiltonian.
\end{proposition}

\begin{proof}
	Recall that the orientations on~$M$ and~$M_{\red}$ are given by~$\omega$ so that $\omega^3$, respectively~$\omega_{\red}^2$, represents a positive volume form on~$M$, respectively~$M_{\red}$. Let $M^{S^1}=\coprod_{i=1}^k X_i$ where $X_i$ are connected surfaces. Since the reduced spaces are all diffeomorphic to $\Sigma_g \times S^2$, for each $t\in \mu(M) \subset S^1$, we can identify $\mu^{-1}(t) / S^1$ with $(\Sigma_g \times S^2, \omega_t)$ where $\omega_t$ is the reduced symplectic form at level~$t$. Let $u$ and $v$ be the basis of $H^2(M_{\red},\Z)$ as chosen in Section~\ref{symplectic ruled surface} after Theorem~\ref{LL}. Then $u^2 = 0$, $uv = 1$, and $v^2=0$. Now the cohomology class of $\omega_t$ can be written as $[\omega_t] = c_t u + d_t v \in H^2(M_{\red},\R)$ for $c_t$,~$d_t \in \R$. Note that if $\mu$ is surjective, as was mentioned in Section~\ref{symplectic ruled surface}, $c_t$ and $d_t$ are continuous on $S^1$ if $M_{\red} \ncong S^2 \times S^2$, and they may be discontinuous at at most one point in~$S^1$ if $M_{\red} \cong S^2 \times S^2$. Since $\int_{M_t}[\omega_t]^2>0$, we have $c_t d_t > 0$ for all $t\in \textrm{int}({\mu(M)}) \subset S^1$. Therefore $c_t$ and $d_t$ are non-vanishing in the interior of~$\mu(M)$, and since $\mu(M)$ is connected, the signs of $c_t$ and $d_t$ in $\textrm{int}({\mu(M)})$ do not change. By Theorem~\ref{LL}, $(M_t,\omega_t)$ is symplectomorphic to some $(M_{\red},\omega')$ which is compatible with the fiber and $[\omega_t] = [\omega'] \in H^2(M_{\red})$, so $\int_{pt \times S^2} \omega_t = d_t > 0$. Therefore we get
\begin{equation}\label{form}
	c_t > 0,\quad d_t > 0,\quad \forall t\in \textrm{int}(\mu(M)).
\end{equation}

	For a fixed component~$X_i$, let $D(X_i) = a_iu + b_iv$ where $a_i, b_i \in \Z$. Let $s=\mu(X_i)$ be a critical value. Then $X_i$ is a symplectically embedded surface in $(\Sigma_g \times S^2, \omega_s)$. By Poincar\'{e} duality,
\begin{equation}\label{pos}
	\int_{X_i} \omega_s  = \int_{M_{\red}} D(X_i) \cdot \omega_s
	= a_id_s + b_i c_s> 0.
\end{equation}
In particular, at least one of $a_i$ and $b_i$ should be positive.

	By \eqref{chern}, the first Chern class is $c_1(M_{\red}) = (2-2g)u + 2v$. By the adjunction formula~\eqref{adjunc}, $D(X_i)^2 - c_1(M_{\red})\cdot D(X_i) = 2g_i -2$ where $g_i$ is the genus of the embedded surface~$X_i$. Hence
\[(a_i u + b_i v)^2 - (a_i u+b_i v)((2-2g)u+2v) + 2 = 2g_i.\]
Therefore, we get
\begin{equation}\label{ad}
	(a_i-(1-g))(b_i-1) = g_i-g.
\end{equation}
	\\
	\textbf{CASE 1: $g \geq 1$}
	\\

	If $g \neq g_i$, by \eqref{ad}, $b_i \neq 1$ and $a_i = 1 + \frac{g_i}{b_i-1} - \frac{b_i}{b_i-1}g$. If we assume $b_i<0$, then $a_i<1$ and thus $a_i$ is non-positive, which contradicts \eqref{pos}. Therefore, $b_i \geq 0$. If $g = g_i$, then $a_i = 1-g$ or $b_i = 1$. Since $g \geq 1$, we have $a_i \leq 0$ or $b_i = 1$. By~\eqref{pos}, $b_i \geq 0$. In any case, $b_i$ is non-negative for each $i$. If we assume that the given action is non-Hamiltonian, then $b_i=0$ for all $i$ by~\eqref{zerosum}. Hence by~\eqref{pos}, all $a_i$ should be positive, which contradicts \eqref{zerosum}.
	\\
	\\
	\textbf{CASE 2: $g = 0$}
	\\

	In this case, $M_{\red} \cong S^2 \times S^2$. As discussed in Section~\ref{symplectic ruled surface}, $M$ can be twisted. So, the parametrization of $\omega_t$ may not be continuous at one regular value $r_0 \in S^1$. Define a function $f$ on $\mu(M)$ by
\begin{equation}\label{function}
f(t) = \frac{d}{dt}[\min(c_t,d_t)].
\end{equation}
Even though $c_t$ and $d_t$ might not be continuous at~$r_0$ for twisted~$M$, $f$ is well-defined at $r_0$. Note that from \eqref{DH}, we have
\begin{equation}\label{funct}
	-e_t = \frac{d}{dt}c_tu + \frac{d}{dt}d_tv.
\end{equation}
We claim that $f$ is a non-increasing step function defined on a finite complement of $\mu(M)$. This is divided into Lemma~\ref{lemma1} to~\ref{lemma5}

\begin{lemma}\label{lemma1}
	The function~$f$ is locally constant, and defined on $\mu(M)$ except possibly at finitely many points. In particular if $f$ is not defined at a point~$r$ in a regular interval $(\alpha, \beta) \subset \mu(M)$, then $f$ is well-defined on $(\alpha, \beta) - \{r\}$.
\end{lemma}

\begin{proof}
	The functions $c_t$ and $d_t$ are linear on~$t$ in every regular interval of $\mu(M) \subset S^1$ by the Duistermaat-Heckman theorem~\eqref{DH}. Hence the derivatives $c_t'$ and $d_t'$ are locally constant in a regular interval, and so is~$f$. Therefore in a regular interval $(\alpha, \beta) \subset \mu(M)$, either $c_t \equiv d_t$, or $c_{t} = d_{t}$ possibly at only one point~$r$ in $(\alpha, \beta)$. Therefore the function $\min\{c_t, d_t\}$ is differentiable on $(\alpha, \beta)$ except possibly at one point~$r$ where $c_{r} = d_{r}$, but $c_t \not\equiv d_t$.

	On the other hand, $c_t$ and $d_t$ may not be differentiable at critical values of $\mu$. Therefore there are only finitely many points of $\mu(M)$ where $f(t) = \frac{d}{dt}[\min(c_t, d_t)]$ is not defined.
\end{proof}

\begin{lemma}\label{lemma2}
	If $f(t)$ is not defined at a regular value~$r$ of~$\mu$, then for a sufficiently small $\epsilon$,
\[f(r - \epsilon) > f(r + \epsilon).\]
\end{lemma}

\begin{proof}
	Let $r$ be a regular value on which $f$ is not defined. For a small~$\epsilon$, choose $t_1 \in (r-\epsilon,r)$ and $t_2 \in (r, r+\epsilon)$. If $\min(c_t,d_t) = c_t$ on $(r-\epsilon,r)$, then $\min(c_t,d_t)$ should be $d_t$ on $(r,r+\epsilon)$. This happens when the $d_t '$ is smaller than $c_t'$ on $(r-\epsilon,r+\epsilon)$. This means that $f(t_1) = c_t' > f(t_2) = d_t '$. Therefore $f$ decreases when it passes through the level $\mu^{-1}(r)$. Similarly, we can prove that $f$ decreases when it passes through the level $\mu^{-1}(r)$ in the case when $\min(c_t,d_t) = d_t$.
\end{proof}

	Let $s$ be a critical value of~$\mu$, and let $M^{S^1} \cap \mu^{-1}(s) = \coprod_{i=1}^k X_i$ where $X_i$ is a fixed surface with genus~$g_i$ for $i=1, \dots, k$. By~\eqref{DH} and \eqref{class}, we can easily see that
\begin{equation}\label{formula}
	e_{s+\epsilon} = e_{s-\epsilon} + (\sum_{i=1}^k a_i)u + (\sum_{i=1}^k b_i)v.
\end{equation}
where $e_{s\pm \epsilon}$ is the Euler class of the principal $S^1$-bundle over $M_{s \pm \epsilon}$ for a sufficiently small $\epsilon$. Therefore by~\eqref{funct},
\begin{equation}\label{formula2}
	c_{t+\epsilon}' = c_{t-\epsilon}' - (\sum_{i=1}^k a_i), \mbox{ and } d_{t+\epsilon}' = d_{t-\epsilon}' - (\sum_{i=1}^k b_i).
\end{equation}

	By equation~\eqref{ad}, we have
\begin{equation}\label{product}
	(a_i - 1)(b_i - 1) = g_i
\end{equation}
for all $i$. So if $X_i$ has genus $g_i > 0$, then
\begin{equation}\label{genus}
	a_i \geq 2, \hskip5mm b_i \geq 2
\end{equation}
since one of $a_i$ and $b_i$ is positive by~\eqref{pos}.
If $g_i = 0$,
\begin{equation}\label{genus0}
	a_i =1 \hskip5mm   \textrm{or}  \hskip5mm  b_i = 1.
\end{equation}

	Now, assume that $f$ is not defined at~$s$. Then we have three possibilities; $c_s < d_s$, $c_s = d_s$, or $c_s > d_s$. Choose a sufficiently small $\epsilon > 0$ so that $f$ is defined on $s-\epsilon \leq t < s$ and $s < t \leq s+\epsilon$.

\begin{lemma}\label{lemma3}
	Assume $c_s < d_s$ so that $f(t) = c_t'$ on $(s-\epsilon, s+\epsilon) - \{s\}$. Then $$ f(s - \epsilon) \geq f(s + \epsilon).$$ The equality holds if and only if $g_i = 0$ and $a_i=0, b_i = 1$ for all $i$. If $g_i>0$ for some $i$, then $f$ decreases by at least $2$ when it passes through the critical value $s$.
\end{lemma}

\begin{proof}
	If $g_i > 0$, then $a_i, b_i \geq 2$ by~\eqref{genus}. If $g_i = 0$, then $a_i = 1$ or $b_i = 1$ by \eqref{genus0}. If $b_i = 1$ and $a_i < 0$, by~\eqref{pos},
\[0 < a_id_s + b_ic_s = a_id_s + c_s < a_id_s + d_s < d_s(a_i + 1) \leq 0,\]
which is a contradiction. So, if $b_i = 1$, then $a_i$ must be non-negative. Therefore if $c_s < d_s$, then $a_i$ is non-negative for all $i$. Note that by~\eqref{formula2}, $f$ decreases by $\sum_{i=1}^k a_i$ when it passes through the level $\mu^{-1}(s)$. Hence we are done.
\end{proof}

	Similarly, we have
\begin{lemma}\label{lemma4}
	Assume $c_s > d_s$ so that $f(t) = d_t'$ on $(s-\epsilon, s+\epsilon) - \{s\}$. Then
\[f(s - \epsilon) \geq f(s + \epsilon).\]
The equality holds if and only if $g_i = 0$ and $a_i=1, b_i = 0$ for all $i$. If $g_i>0$ for some $i$, then $f$ decreases by at least $2$ when it passes through the critical value~$s$.
\end{lemma}

\begin{remark}\label{negative}
	Different from the case when $g \geq 1$, if $a_i = 1$, then $b_i$ may be negative in Lemma \ref{lemma3}. Similarly, if $b_i = 1$, then $a_i$ may be negative in Lemma \ref{lemma4}.
\end{remark}

\begin{lemma}\label{lemma5}
	Assume $c_s = d_s$. Then
\[f(s - \epsilon) \geq f(s + \epsilon).\]
In particular, $a_i, b_i \geq 0$ and one of $a_i$ and $b_i$ is positive for all $i$. If $g_i>0$ for some $i$, then $f$ decreases by at least $2$ when it passes through the critical value~$s$.
\end{lemma}

\begin{proof}
	If $g_i > 0$, then $a_i, b_i \geq 2$ by~\eqref{genus}. If $g_i = 0$, then $a_i = 1$ or $b_i = 1$ by~\eqref{genus0}. On the other hand, $0< a_id_s + b_ic_s = c_s(a_i + b_i)$ by~\eqref{pos}. Therefore if $a_i = 1$, then $b_i \geq 0$. Similarly, $a_i \geq 0$ if $b_i = 1$. Hence
\begin{equation}\label{eq:pos}
\sum_{i=1}^k a_i \geq 0, \quad \sum_{i=1}^k b_i \geq 0.
\end{equation}
Hence $c_t'$ and $d_t'$ do not increase when they pass through the critical value~$s$ by~\eqref{formula2}.

	Assume that $c_t < d_t$ on $(s-\epsilon, s)$. Then $f(t) = c_t' > d_t'$ on $(s-\epsilon, s)$. Hence $f(t)$ does not increase when it passes through $s$. Similarly, $f$ does not increase in the case when $c_t > d_t$ on $(s-\epsilon, s)$.

	If $c_t = d_t$ on $(s-\epsilon, s)$, since $c_t'$ and $d_t'$ do not increase when they passes through $s$ by \eqref{formula2}~and \eqref{eq:pos}, $f$ does not increase.
\end{proof}

	Now, back to the proof of Proposition~\ref{thm1}. By the previous lemmas, $f$ is a non-increasing step function on the complement of a finite set in $\mu(M)$. Assume the action is non-Hamiltonian so that $\mu$ is surjective. Then $f$ is defined on $S^1$ except at finitely many points. Since $f$ is a non-increasing function, $f$ should be constant. By Lemma~\ref{lemma3}, Lemma~\ref{lemma4}, and Lemma \ref{lemma5}, $g_i = 0$ and $a_i+b_i >0$ for all $i$. This contradicts \eqref{zerosum2}. Hence the $S^1$-action should be Hamiltonian. This proves Proposition~\ref{thm1}.
\end{proof}

    Now, let $(M,\omega)$ be a $6$-dimensional closed Hamiltonian semifree $S^1$-manifold with moment map $\mu : M \rightarrow [t_{\min},t_{\max}] \subset \R$. Assume that $M^{S^1}$ consists of surfaces. Let $\Sigma_{\min}$ and $\Sigma_{\max}$ be two fixed components such that $\mu(\Sigma_{\min}) = t_{\min}$ and $\mu(\Sigma_{\max}) = t_{\max}$. Note that $\Sigma_{\min}$ and $\Sigma_{\max}$ are diffeomorphic to the Riemann surface $\Sigma_g$ of genus $g$. We give a basis of $H^2(M_{\red}, \Z)$ as follows. By Morse theory, the level set $\mu^{-1}(r)$ near the minimum of $\mu$ is an $S^3$-bundle over $\Sigma_{\min}$ and $S^1$ acts on $\mu^{-1}(r)$ fiberwise. Therefore the $S^3$-bundle $\mu^{-1}(r) \rightarrow \Sigma_{\min}$ induces an $S^2$-fibration $\pi : M_{\red} \rightarrow \Sigma_{\min}$. By Theorem~\ref{LL}, we may assume the reduced symplectic form $\omega_r$ to be compatible with the ruling given by $\pi$. Let $u$ be dual to the class represented by a symplectic fiber $F$. When $M_{\red}$ is trivial, let $v$ be dual to the class represented by the base $B = [\Sigma_g \times pt]$ so that $\int_{M_{\red}} uv = 1$. For the non-trivial case, let $v$ be dual to the class represented by a section with self-intersection $-1$ with $\int_{M_{\red}} uv = 1$. Then $u$ and $v$ form a basis of $H^2(M_{\red}, \Z)$ and we can express $[\omega_t] = c_tu + d_tv$ for all $t \in (t_{\min}, t_{\max})$.

\begin{remark}\cite{L1}\label{twist}
	Consider the $S^3$-fibration $\mu^{-1}(r) \rightarrow \Sigma_{\max}$ near the maximum of $\mu$. It also induces an $S^2$-fibration $\pi' : M_{\red} \rightarrow \Sigma_{\max}$. Similarly, we can choose a basis $u', v'$ of $H^2(M_{\red}, \Z)$ with respect to $\pi'$ as above. If $M_{\red} \ncong S^2 \times S^2$, then we can show that $u=u'$ and $v=v'$ in $H^2(M_{\red}, \Z)$. In the case when $M_{\red} \cong S^2 \times S^2$, there are only two possibilities; $u=u'$, $v=v'$, or $u=v'$ and $v=u'$. If $u=u'$ and $v=v'$, we say that $M$ is \textit{not twisted}. If $u=v'$ and $v=u'$, we say that $M$ is \textit{twisted}. If $M$ is twisted, then the isotopy between the reduced spaces near the minimum and the maximum along the Morse flow maps a fiber of $M_{\red} \rightarrow \Sigma_{\min}$ representing the dual of $u$ to a section of $M_{\red} \rightarrow \Sigma_{\max}$ representing the dual of $v'$, and a section of $M_{\red} \rightarrow \Sigma_{\min}$ representing the dual of $v$ to a fiber of $M_{\red} \rightarrow \Sigma_{\max}$ representing the dual of $u'$. See \cite{L1} for more details.
\end{remark}

	To prove the second statement of Theorem~\ref{main} in the trivial fibration cases, we need the following lemma due to Li.

\begin{lemma}\label{Li}\cite{L1}
	Let $(M,\omega)$ be a $6$-dimensional closed Hamiltonian semifree $S^1$-manifold with moment map $\mu : M \rightarrow \R$.

	\begin{itemize}
	\item[i)] If the minimum is a surface $\Sigma_g$ with the first Chern number of its normal bundle equal to $b_{\min}$, then the reduced space $M_{\red}$ near the minimum is diffeomorphic to $\Sigma_g \times S^2$ if and only if $b_{\min} = 2k$ is even, and it is diffeomorphic to $E_{\Sigma_g}$ if and only if $b_{\min} = 2k+1$ is odd. In either case, the Euler class of the principal $S^1$-bundle over $M_{\red}$ near the minimum is $ku - v$.

	\item[ii)] If the maximum is a surface $\Sigma_g$ with the first Chern number of its normal bundle equal to $b_{\max}$, then the reduced space $M_{\red}$ near the maximum is diffeomorphic to $\Sigma_g \times S^2$ if and only if $b_{\max} = 2k'$ is even, and it is diffeomorphic to $E_{\Sigma_g}$ if and only if $b_{\max} = 2k'+1$ is odd. In either case, the Euler class of the principal $S^1$-bundle over $M_{\red}$ near the maximum is $-k'u + v$ if $M_{\red} \ncong S^2 \times S^2$ near the maximum. In the case when $M_{\red} \cong S^2 \times S^2$ near the maximum, the Euler class of the principal $S^1$-bundle over $M_{\red}$ near the maximum is $-k'u + v$ or $-k'v + u$.
	\end{itemize}

\end{lemma}

\begin{proposition} \label{Euler}
	Let $M$ be a $6$-dimensional closed Hamiltonian semifree $S^1$-manifold whose fixed point set consists of surfaces, and let $\mu : M \rightarrow [t_{\min},t_{\max}] \subset \R$ be the moment map. Suppose $M_{\red}$ is diffeomorphic to $\Sigma_g \times S^2$. Then the number of the fixed surfaces with positive genera is at most four. If the maximal fixed surface is a sphere, then there is at most one fixed surface with positive genus. Moreover if there are four fixed surfaces with positive genera, these genera are all equal to $g$.
\end{proposition}

\begin{proof}
	Let $M^{S^1} = \coprod_i X_i$ be the set of fixed surfaces, and let $g_i$ be the genus of $X_i$ for each $i$. Let $D(X_i) = a_iu + b_iv$ for each $i$, where $\{u,v\}$ is the basis of $H^2(M_{\red}, \Z)$ as mentioned before Remark~\ref{twist}. Let $\Sigma_{\textrm{mid}}$ denote the subset of $M^{S^1}$ consisting of non-extremal fixed surfaces, $\Sigma_{\min}$ the minimal fixed surface, and $\Sigma_{\max}$ the maximal fixed surface.

	We first prove the case when $g \geq 1$. Then $b_i \geq 0$ for all $i$ by the proof of Proposition~\ref{thm1}. For a principal $S^1$-bundle $\mu^{-1}(r) \rightarrow M_{\red}$ for a regular value $r \in S^1$, let $e_{\min}$ be the Euler class of principal $S^1$-bundle near the minimum of $\mu$ and $e_{\max}$ be the one near the maximum of $\mu$. By~\eqref{class},
\[e_{\max} = e_{\min} + \sum_{X_i \subset {\Sigma}_{\textrm{mid}}} D(X_i).\]
By Lemma~\ref{Li}, $e_{\min} = ku-v$ and $e_{\max} = -k'u + v$ for some $k, k' \in \Z$. Hence
\begin{equation}\label{two}
	\sum_{X_i \subset {\Sigma}_{\textrm{mid}}}  b_i = 2.
\end{equation}
Therefore, the possible values of $b_i$ are $0$,~$1$, or~$2$. If $b_i = 0$, then \eqref{ad} implies that $a_i = 1 - g_i$. Since $a_i$ and $b_i$ cannot be both non-positive by \eqref{pos}, $a_i$ should be positive so that $g_i = 0$. If $b_i = 1$, then $g_i = g$ by \eqref{ad}. If $b_i = 2$, then $b_j = 0$ for all $j \neq i$ by \eqref{two}. Therefore there are at most two fixed surfaces $X_i$ which have $b_i \neq 0$ and are neither maximal nor minimal with respect to $\mu$. So, the number of fixed surfaces which are interior with respect to $\mu$ and have positive genera is at most two. Counting with the minimal and the maximal fixed surfaces, there are at most four fixed surfaces with positive genera. In particular, if there are four fixed surfaces having non-zero genera, then these genera are all equal to $g$.

	We now prove the case when $g = 0$. In this case, $M_{\red} \cong S^2 \times S^2$ and $b_i$ may be negative as we mentioned in Remark~\ref{negative}. So we need another approach. Remember that $[\omega_t] = c_tu + d_tv$. Let $F \cong pt \times S^2$ be a fiber of $M_{\red} \rightarrow \Sigma_{\min}$ which represents the dual of $u$, and let $B \cong S^2 \times pt$ be a section of $M_{\red} \rightarrow \Sigma_{\min}$ which represents the dual of $v$ as mentioned before Remark \ref{twist}. Since the volume of the fiber tends to 0 and the volume of the section $B$ tends to the symplectic volume of the minimal fixed surface as $t$ approaches $t_{\min}$,
\[0 = \lim_{t \rightarrow t_{\min}} \int_F \omega_t = \lim_{t \rightarrow t_{\min}} \int_{M_{\red}} [\omega_t] \cdot u = \lim_{t \rightarrow t_{\min}} d_t\]
and
\[0 < \lim_{t \rightarrow t_{\min}} \int_B \omega_t = \lim_{t \rightarrow t_{\min}} \int_{M_{\red}} [\omega_t] \cdot v = \lim_{t \rightarrow t_{\min}} c_t.\]
So, $c_t > d_t$ near $\Sigma_{\min}$. Then, near $\Sigma_{\min}$, the value of the function $f$ defined in~\eqref{function} is equal to the negative of the coefficient of $v$ of $e_{\min}$ by~\eqref{funct}. By Lemma~\ref{Li}, $e_{\min} = ku - v$ so that $f(t) = 1$ near $\Sigma_{\min}$. Near $\Sigma_{\max}$, let $F' \cong pt \times S^2$ be a fiber of $M_{\red} \rightarrow \Sigma_{\max}$ which represents the dual of $u'$ and let $B' \cong S^2 \times pt$ be a section of $M_{\red} \rightarrow \Sigma_{\max}$ which represents the dual of $v'$ as mentioned in Remark~\ref{twist}.

	If $M$ is twisted, then $u = v'$ and $v = u'$ so that $B'$ represents the dual of $u$ and $F'$ represents the dual of $v$. Therefore
\[0 = \lim_{t \rightarrow t_{\max}} \int_{F'} \omega_t = \lim_{t \rightarrow t_{\max}} \int_{M_{\red}} [\omega_t] \cdot v = \lim_{t \rightarrow t_{\max}} c_t\]
and
\[0 < \lim_{t \rightarrow t_{\max}} \int_{B'} \omega_t = \lim_{t \rightarrow t_{\max}} \int_{M_{\red}} [\omega_t] \cdot u = \lim_{t \rightarrow t_{\max}} d_t.\] Hence $f(t) = \frac{d}{dt}c_t$ which is equal to the negative coefficient of $u$ of $e_{\max}$ near $M_{\max}$. By Lemma~\ref{Li}, $e_{\max} = u-k'v$ for some integer $k'$ so that $f(t) = -1$ near $M_{\max}$.

	If $M$ is not twisted, then $u = u'$ and $v = v'$ so that $B'$ represents the dual of $v$ and $F'$ represents the dual of $u$. Therefore
\[0 = \lim_{t \rightarrow t_{\max}} \int_{F'} \omega_t = \lim_{t \rightarrow t_{\max}} \int_{M_{\red}} [\omega_t] \cdot u = \lim_{t \rightarrow t_{\max}} d_t\]
and
\[0 < \lim_{t \rightarrow t_{\max}} \int_{B'} \omega_t = \lim_{t \rightarrow t_{\max}} \int_{M_{\red}} [\omega_t] \cdot v = \lim_{t \rightarrow t_{\max}} c_t.\] Hence $f(t) = \frac{d}{dt}d_t$ which is equal to the negative of the coefficient of $v$ of $e_{\max}$ near $\Sigma_{\max}$. By Lemma~\ref{Li}, $e_{\max} = -k'u + v$ for some integer $k'$ so that $f(t) = -1$ near $\Sigma_{\max}$. So if $M_{\red} \cong S^2 \times S^2$, then $f(t) = 1$ near $\Sigma_{\min}$ and $f(t) = -1$ near $\Sigma_{\max}$.

	Note that $f$ is a non-increasing function, and if the fixed surface $X_i$ has genus $g_i >0$, then $f$ decreases by at least 2 when it passes through the critical level containing $X_i$ by Lemma~\ref{lemma3},~\ref{lemma4}, and~\ref{lemma5}. Hence there exists at most one fixed surface $X_i$ with $g_i \neq 0$. Since $\Sigma_{\min}$ and $\Sigma_{\max}$ are spheres, there is at most one interior fixed surface $X_i$ with $g_i \neq 0$.

	To sum up, if $M_{\red} \cong \Sigma_{g} \times S^2$ with $g \geq 1$, then there are at most four fixed surfaces having non-zero genera in $M$. In particular, if there are four fixed surfaces having non-zero genera, then these genera are all equal to $g$. If $M_{\red} \cong S^2 \times S^2$, then there is at most one fixed surface having non-zero genus.
\end{proof}

	In fact, the number of fixed surfaces with positive genera cannot be two. This will be proved in the next section. See Proposition~\ref{twopositivegenus}.

\section{Proof of Theorem \ref{main} : non-trivial fibration cases}\label{proof for nontrivial}

	Let $E_{\Sigma_g}$ denote a non-trivial $S^2$-bundle over a Riemann surface $\Sigma_g$ of genus~$g$. In this section, we assume that $M_{\red} \cong E_{\Sigma_g}$. Let $u$ and $v$ be the basis of $H^2(M_{\red},\Z)$ chosen as in Section~\ref{symplectic ruled surface} so that $\int_{M_{\red}} u^2 = 0$, $\int_{M_{\red}} uv = 1$ and $\int_{M_{\red}} v^2 = -1$. Let $M^{S^1} = \coprod_i X_i$ be the fixed point set consisting of surfaces, and let $D(X_i) = a_i u + b_i v$ for each $i$. By Proposition~\ref{attach}, we can parametrize the classes of reduced symplectic forms by $t \in S^1$ so that $[\omega_t] = c_t u + d_t v$, where $c_t$ and $d_t$ are continuous for all $t \in \mu(M)$. By Theorem~\ref{LL}, we may assume that the fiber $F \cong S^2$ is symplectic. Hence $\int_F \omega_t = d_t > 0$.

	Before we prove Theorem~\ref{main}, we make some remark on symplectic $4$-manifolds. A symplectic $4$-manifold $(M,\omega)$ is called \textit{minimal} if it has no symplectically embedded sphere with self-intersection~$-1$. It is well-known that any irrational ruled surface is minimal.

	The following theorem due to McDuff tells us which cohomology classes in $H^2(M_{\red},\Z)$ can be represented by a symplectic form when $M_{\red}$ is a non-trivial rational ruled surface, i.e., a Hirzebruch surface.

\begin{theorem}(\cite{MD2} ,\cite{McS} p.204)\label{McD}
	Let $M_{\red}$ be a non-trivial rational ruled surface. For the basis $\{u,v\}$ chosen above, there is a symplectic form $\omega$ on $M_{\red}$ which is compatible with the given ruling with $[\omega] = cu+dv$ if and only if $c>d>0$.
\end{theorem}

Now, we prove the first part of Theorem~\ref{main} for the non-trivial cases.

\begin{proposition}\label{lastprop}
	Let $(M,\omega)$ be a 6-dimensional closed symplectic semifree $S^1$-manifold with non-empty fixed point set $M^{S^1}$ consisting of surfaces. Let $\mu : M \rightarrow S^1$ be a generalized moment map. If $M_{\red}$ is diffeomorphic to $E_{\Sigma_g}$, then the action is Hamiltonian.
\end{proposition}

\begin{proof}
	Let $M^{S^1} = \coprod_i X_i$ be the fixed surfaces and $g_i \geq 0$ be the genus of $X_i$ for each $i$. Since the reduced space for $\mu$ is diffeomorphic to $E_{\Sigma_g}$ for each $t\in \mu(M)$, we can identify $\mu^{-1}(t) / S^1$ with $(E_{\Sigma_g}, \omega_t)$ where $\omega_t$ is the reduced symplectic form at level $t$. Remember that $D(X_i) = a_iu + b_iv$ denotes the class represented by the Poincar\'{e} dual of $X_i$ in $H^2({E_{\Sigma_g}},\Z)$.

	Since $\omega_t$ is the reduced symplectic form for every $t \in \mu(M) \subset S^1$, we have $\int_{E_{\Sigma_g}}[\omega_t]^2 > 0$. For each $i$, $X_i$ is a symplectic submanifold of the reduced space at critical value $s = \mu(X_i)$. So we have $\int_{E_{\Sigma_g}} [\omega_s] \cdot D(X_i) > 0$. Since $\int_F \omega_t = d_t > 0$, we have
\begin{equation}\label{pp}
	\int_{E_{\Sigma_g}}(c_t u + d_t v)^2 = (2c_t - d_t)d_t >  0
	\Longleftrightarrow 2c_t > d_t > 0,
\end{equation}

\begin{equation}\label{sub}
	\int_{E_{\Sigma_g}}(c_s u + d_s v)(a_i u + b_i v) = a_i d_s + b_i c_s - b_i d_s> 0.
\end{equation}

	Recall that $c_1(E_{\Sigma_g}) = (3-2g)u + 2v$ by \eqref{chern2}. The adjunction formula for $X_i$ implies that
\[(a_i u + b_i v)^2 - (a_i u+b_i v)((3-2g)u + 2v) + 2 = 2g_i.\]
	Therefore
\begin{equation}\label{ad3}
\left(a_i-\frac{b_i}{2}-(1-g)\right)(b_i-1) = g_i - g.
\end{equation}
	If we use the substitution
\[\alpha_i = a_i - \frac{b_i}{2}, \quad \beta_i = b_i,\]
\[\gamma_t =c_t-\frac{d_t}{2}, \quad \delta_t = d_t,\]
the above expressions \eqref{pp},~\eqref{sub}, and~\eqref{ad3} are written in the following familiar forms. (Compare with~\eqref{form}, \eqref{pos}, and~\eqref{ad}).
\begin{equation}\label{newform}
	\gamma_t > 0,\quad \delta_t > 0, \quad \forall t\in \mu(M),
\end{equation}
\begin{equation}\label{newpos}
	\alpha_i\delta_s + \beta_i\gamma_s > 0,
\end{equation}
\begin{equation}\label{newad}
	(\alpha_i - (1-g))(\beta_i - 1) = g_i - g.
\end{equation}

	In particular, \eqref{newform} and \eqref{newpos} imply that $\alpha_i$ and $\beta_i$ cannot be both non-positive simultaneously. Since $\alpha_i \in \frac{1}{2}\Z$ and $\beta_i \in \Z$, the situation is different from the previous section.
	\\
	\\
	\textbf{CASE 1: $g \geq 1$}
	\\

	If $g_i \neq g$, by \eqref{newad} $\beta_i \neq 1$ and $\alpha_i = 1 + \frac{g_i}{\beta_i-1} - \frac{\beta_i}{\beta_i-1}g$. Assume $\beta_i < 0$. Then $\frac{g_i}{\beta_i - 1} \leq 0$ and $\frac{-g\beta_i}{\beta_i - 1} < 0$. Since $\alpha_i$ should be positive by \eqref{newpos}, the only possible value for $\alpha_i$ is $\frac{1}{2}$. Therefore, the equation~\eqref{newad} is written as
\[\beta_i = \frac{2g_i-1}{2g-1}.\]
	Since $\beta_i$ is a negative integer and $g \geq 1$, we have $g_i=0$ and $g=1$. So the only possible case is $\alpha_i = \frac{1}{2}$ and $\beta_i = -1$. Therefore $a_i = 0$ and $b_i = -1$ so that $X_i$ is a sphere with $D(X_i) = -v$ and is symplectically embedded in $E_{\Sigma_1}$. This means that $X_i$ is an exceptional sphere in $E_{\Sigma_1}$. Since any irrational ruled surface is minimal, it is a contradiction. Therefore, $\beta_i \geq 0$.

If $g_i = g$, then $\alpha_i = 1-g$ or $\beta_i = 1$. Since $g \geq 1$, we have $\alpha_i \leq 0$ or $\beta_i = 1$. By~\eqref{newpos}, $\beta_i \geq 0$. Hence if $g \geq 1$, then $b_i = \beta_i \geq 0$ for all $i$.

	Now suppose that the action is not Hamiltonian. By~\eqref{zerosum}, $b_i = 0$ for all $i$. So $\alpha_i = a_i - \frac{b_i}{2} = a_i$ should always be positive for all $i$ by~\eqref{newpos}. It contradicts~\eqref{zerosum}.
	\\
	\\
	\textbf{CASE 2: $g = 0$}
	\\

	Note that Theorem~\ref{LL} allows us to assume $\omega$ is compatible with the ruling. Hence by Theorem~\ref{McD}, we have $c_t > d_t > 0$ for all $t \in \mu(M) \subset S^1$. By~\eqref{sub}, we have
\begin{equation}\label{last}
	a_i d_s + b_i (c_s - d_s) > 0,
\end{equation}
which implies that $a_i > 0$ or $b_i > 0$ for each $i$. By~\eqref{newad}, we have
\begin{equation}\label{adg}
	(\alpha_i-1)(\beta_i-1) = g_i.
\end{equation}

	In the case when $g_i = 0$, we have $\alpha_i = 1$ or $\beta_i = 1$. If $\alpha_i = 1$ and $b_i = \beta_i < 0$, then $a_i = \frac{b_i}{2} + \alpha_i < 1$. Since $a_i$ is an integer, $a_i \leq 0$ which contradicts~\eqref{last}. Therefore $b_i = \beta_i \geq 0$.

	We now consider the case when $g_i > 0$. Assume $\beta_i = b_i < 0$. Then $a_i$ should be positive by~\eqref{last}. But by~\eqref{adg}, $\alpha_i < 1$. Hence $a_i = \frac{b_i}{2} + \alpha_i < 1$. Since $a_i$ is an integer, it cannot be positive. It contradicts~\eqref{last}.

	Therefore $b_i \geq 0$ for all $i$. Now assume the action is non-Hamiltonian. Then $b_i = 0$ for all $i$ by~\eqref{zerosum}, and $a_i > 0$ for all $i$ by~\eqref{last}. It contradicts~\eqref{zerosum} again.
\end{proof}

\begin{proposition}\label{number}
	Let $M$ be a $6$-dimensional closed Hamiltonian semi-free $S^1$-manifold whose fixed point set consists of surfaces. Suppose $M_{\red}$ is diffeomorphic to $E_{\Sigma_g}$. Then the number of fixed surfaces with positive genera is at most four. If the maximal fixed surface is a sphere, then so is the minimal one and there is at most one fixed surface with positive genus. If there are four fixed surfaces with positive genera, then these genera are all equal to $g$.
\end{proposition}

\begin{proof}
	Let $e_{\min}$ be the Euler class of the principal $S^1$-bundle $\mu^{-1}(r) \rightarrow M_{\red}$ for some regular value $r$ near the minimum value of $\mu$, and let $e_{\max}$ be the one near the maximum value of $\mu$. By Lemma~\ref{Li}, the Euler class of the principal $S^1$-bundle over $M_{\red}$ near the minimum is $ku - v$ and the Euler class near the maximum is $-k'u + v$ for some integers $k$ and $k'$. By the proof of Proposition~\ref{lastprop}, $b_i \geq 0$ for all $i$. Moreover $\sum_i b_i = 2$ by \eqref{two}. So the only possible values of $b_i$ are $0$,~$1$ or~$2$.

	Consider the case when $g \geq 1$ first. If $b_i = \beta_i = 0$, then \eqref{newad} implies that $\alpha_i = 1 - g_i$. Since $\alpha_i$ and $\beta_i$ cannot be both non-positive by~\eqref{newpos}, $\alpha_i$ should be positive so that $g_i = 0$. If $b_i = \beta_i = 1$, then $g_i = g$ by~\eqref{newad}. If $b_i = 2$, then $b_j = 0$ for all $j \neq i$ by~\eqref{two}. By the previous argument, $b_j=0$ implies $g_j = 0$, and hence all $X_j$'s are spheres for $j \neq i$.

	In the case when $g = 0$, the minimum and the maximum are $2$-spheres. If $g_i \geq 1$, then $\beta_i = b_i \geq 2$ by~\eqref{adg}. Therefore $b_i = 2$ and the other $b_j$'s are zero for $j \neq i$.

	So, the number of interior fixed surfaces which have positive genera is at most two. With the minimal and the maximal fixed surfaces, there are at most four fixed surfaces with positive genera, and these genera are all equal to $g$. Moreover, if the minimal and the maximal surfaces are spheres, then there is at most one fixed surface with non-zero genus.
\end{proof}

	The following proposition says that the number of fixed surfaces with positive genera cannot be two. This finishes the proof of Theorem~\ref{main}.

\begin{proposition}\label{twopositivegenus}
	There is no $6$-dimensional closed semifree Hamiltonian $S^1$-manifold such that $M^{S^1}$ consists of surfaces and the number of surfaces with positive genera is two.
\end{proposition}

\begin{proof}
	Assume that there exists a closed semifree Hamiltonian $S^1$-manifold such that $M^{S^1}$ consists of surfaces and the number of surfaces with positive genera is two. Let $\mu : M \rightarrow \R$ be the moment map with respect to the given action. If the maximal fixed set is a sphere, then there is at most one fixed surface with positive genus by Proposition~\ref{Euler} and Proposition~\ref{number}. So we assume that the genus $g$ of the extremal fixed surfaces is positive and the interior fixed surfaces are all spheres. By Lemma~\ref{Li}, the Euler class of the principal $S^1$-bundle $\mu^{-1}(r) \rightarrow M_{\red}$ near the minimum is $ku-v$ and the Euler class near the maximum is $-k'u + v$ for some integers $k$ and $k'$. Since
\[e_{\max} = e_{\min} + \sum_{X_i \subset {\Sigma}_{\textrm{mid}}} D(X_i)\]
by
\eqref{class}, there is an interior fixed sphere $Z$ whose dual class is $au + bv$ in $H^2(M_{\red}, \Z)$ for some integer $b \neq 0$.

	Let $i : Z \hookrightarrow M_{\red}$ be the inclusion and let $\pi : M_{\red} \rightarrow \Sigma_g$ be the given fibration. Then the degree of the composition map $f = \pi \circ i : S^2 \rightarrow \Sigma_g$ is not zero. Indeed, let $\sigma$ be the volume form on $\Sigma_g$. Since $\pi^*\sigma$ is zero on the fiber $F$, we have $\int_{M_{\red}}[\pi^*\sigma]u = \int_F \pi^*\sigma = 0$. Hence $[\pi^*\sigma] = cu$ for some integer $c$. Recall that $v$ is represented by a section $s : \Sigma_g \rightarrow M_{\red}$. Since $s^*\pi^* = (\pi \circ s)^* = id|_{H^2(\Sigma_g)}$, $\pi^*$ is injective. Therefore $c \neq 0$. Now $\int_Z f^*([\sigma]) = \int_Z i^*(\pi^*([\sigma])) = \int_Z i^*(cu) = \int_{M_{\red}} (au+bv)cu = bc \neq 0$. But $f$ is a continuous map $S^2 \rightarrow \Sigma_g$ and hence it has degree zero. To see this, choose $x,y \in H^1(\Sigma_g)$ so that $xy \in H^2(\Sigma_g)$ is a fundamental class. Then $f^*(xy) = f^*(x)f^*(y) = 0$ since $H^1(S^2) = 0$. So we are done.
\end{proof}

\section{Construction of $6$-dimensional examples}\label{construction of example}

	In this section we construct some $6$-dimensional closed Hamiltonian semifree $S^1$-manifolds whose fixed components are all surfaces; Example~\ref{4genus} is the case when $M^{S^1}$ contains four fixed surfaces whose genera are non-zero, and Example~\ref{sphere} is the case when $M^{S^1}$ consists of $N$ fixed spheres for an  arbitrary integer $N \geq 4$. Example~\ref{4genus} implies that the maximal number of fixed surfaces with positive genera is four and it is optimal. Example~\ref{sphere} shows that there is no upper bound on the number of fixed components while there are exactly 8 fixed points in the case when $M^{S^1}$ consists of isolated fixed points. See \cite{TW}.

	Let $M$ be a closed Hamiltonian semifree $S^1$-manifold with moment map~$\mu$. Denote the critical values of $\mu$ by
\[\textrm{minimum} = c_0 <  \dots < c_i < \dots < c_k = \textrm{maximum}.\]
	For sufficiently small $\epsilon>0$, we can divide $M$ into compact submanifolds with boundary as follows:
\[\mu^{-1}[c_i-\epsilon, c_i+\epsilon]\quad\mathrm{: \quad critical \,piece},\]
\[\mu^{-1}[c_i+\epsilon, c_{i+1}-\epsilon]\quad\mathrm{: \quad regular \,piece}.\]
Each piece is a Hamiltonian semifree $S^1$-manifold with boundary. So, we will construct these pieces and glue them together along their boundaries.

	To construct regular pieces, we need the following proposition.
\begin{proposition} (\cite{McS} p.156)\label{regular}
	For a given smooth manifold $B$, let $I \subset \R$ be an interval and $\{\omega_t\}_{t \in I}$ be a family of symplectic forms on $B$ such that $$[\omega_t] = [\omega_r] - (t-r)e$$ for $t,r \in I$ where $e \in H^2(B,\Z)$. Moreover, let $\pi : P \rightarrow B$ be a circle bundle with first Chern class~$e$. Then there is an $S^1$-invariant symplectic form~$\omega$ on the manifold $P \times I$ with moment map~$\mu$ equal to the projection $P \times I \rightarrow I$ whose reduced spaces are $(B, \omega_t)$.
\end{proposition}

	Now, we will construct a regular piece using Proposition~\ref{regular}. The construction is as follows. Let $B = \Sigma_g \times S^2$ and let $I$ be a closed interval in $\R$. Let $\{u,v\}$ be a basis of $H^2(B,\Z)$ as in Section~\ref{symplectic ruled surface}. Then any class $cu + dv$ with $c>0$ and $d>0$ can be represented by a standard split symplectic form $c\pi_1^*\sigma_1 + d\pi_2^*\sigma_2$ where $\pi_1 : \Sigma_g \times S^2 \rightarrow \Sigma_g$ and $\pi_2 : \Sigma_g \times S^2 \rightarrow S^2$ are the canonical projections and $\sigma_1$, $\sigma_2$ are the normalized area forms on $\Sigma_g$ and $S^2$ respectively. Hence if we have $\{\omega_t = c_t\pi_1^*\sigma_1 + d_t\pi_2^*\sigma_2\}_{t \in I}$ with $c_t, d_t > 0$ for $t\in I$ satisfying $[\omega_t] = [\omega_r] - (t-r)e$ for some $e \in H^2(B,\Z)$, then we can construct a Hamiltonian semifree $S^1$-manifold $(P \times I, \omega)$ with moment map~$\mu$ equal to the projection $P \times I \rightarrow I$ and with reduced spaces $(B, \omega_t)$ where $P$ is the principal $S^1$-bundle over~$B$ whose first Chern class is $e \in H^2(B,\Z)$.

	Next we construct critical pieces. First, we construct a Hamiltonian $S^1$-manifold with boundary such that there is only one fixed component~$B$ as a minimal fixed set with respect to the moment map. The following proposition is from~\cite{L2}. We give a proof for reader's convenience.

\begin{proposition}\label{construction}\cite{GS,L2}
	Let $B$ be a compact Riemann surface with symplectic form~$\omega_B$, and let $E \rightarrow B$ be a $\C^2$-bundle with first Chern number $\langle c_1(E), B \rangle = b_{\min} \in \Z$. Then there is a symplectic structure~$\omega_E$ on a neighborhood $E_{\delta}$ of $B$ in~$E$ such that $\omega_E|_B = \omega_B$. Moreover, there is a Hamiltonian $S^1$ action on~$E_{\delta}$, by shrinking it if necessary, such that $B$ is the unique fixed component and have the minimum value of the moment map.
\end{proposition}

\begin{proof}
	Choose a Hermitian metric on~$E$ and consider the unitary frame bundle~$F$ of~$E$. It is a principal $U(2)$-bundle over~$B$. If we choose a connection on~$F$, then it gives a projection $TF \rightarrow VF$ from the tangent bundle of~$F$ to the bundle of vertical tangent vectors, and this gives the dual map $i : V^*F \rightarrow T^*F$. Let $\omega_F$ be the canonical symplectic form on the cotangent bundle~$T^*F$. Then $\omega := i^*\omega_F + \pi^*\omega_B$ is a symplectic form on a neighborhood~$V$ of the zero section of~$V^*F$, where $\pi : V^*F \rightarrow B$ is the composition of projection maps. The cotangent bundle $(T^*F, \omega_F)$ has a canonically lifted Hamiltonian $U(2)$-action, and we may choose $V$ invariant so that the $U(2)$-action on $(V, \omega)$ is Hamiltonian. Let $\Phi : (V \times \C^2, -\omega \oplus \omega_{st}) \rightarrow u(2)^*$ be a moment map of the $U(2)$-action, where $\omega_{st}$ denotes the standard symplectic structure on~$\C^2$, and $U(2)$ acts on $\C^2$ in the canonical way. Define a Hamiltonian $S^1$-action on $(V \times \C^2, -\omega \oplus \omega_{st})$ such that it acts on the first factor trivially and acts on the second factor as the standard diagonal action with a weight $(1,1)$ with respect to the orientation on~$\C^2$ given by $\omega_{st}$.

	On a local trivialization chart $\frak U = \cup \hskip0.5mm U_i$ where $U_i \subset B$, $V^*F$ is locally isomorphic to $U_i \times U(2) \times u(2)^*$ and the moment map for the $U(2)$-action on $U_i \times U(2) \times u(2)^*$ is just the projection onto~$u(2)^*$. Therefore $0 \in u(2)^*$ is a regular value of~$\Phi$ and we can identify $E$ with
\begin{equation}
	E = F \times_{U(2)} \C^2 = \Phi^{-1}(0)/U(2).
\end{equation}
Then the reduced symplectic form~$\omega_E$ defines a symplectic structure on~$E_{\delta}$. Note that the $S^1$-action commutes with the $U(2)$-action on $V \times \C^2$. Hence the $S^1$-action descends to $E_{\delta}$, and the action is Hamiltonian with moment map $\mu(x,z_1,z_2) = \frac{1}{2}(|z_1|^2 + |z_2|^2)$ near the zero section of $U_i \times \C^2$. Therefore $E_{\delta}$ is a Hamiltonian semifree $S^1$-manifold with fixed point set $B$, and the moment map~$\mu$ has the minimum at~$B$.
\end{proof}

\begin{remark}
	By Lemma~\ref{Li}, the Euler class~$e_{\min}$ of the principal $S^1$-bundle $\mu^{-1}(t) \rightarrow \mu^{-1}(t)/S^1$ near the minimum value is determined by the first Chern number~$b_{\min}$. Moreover, the above construction is valid for arbitrary~$b_{\min}$. Therefore we can construct such a critical piece~$E_{\delta}$ with any given~$e_{\min}$.
\end{remark}

	Similarly, we can construct a Hamiltonian $S^1$-manifold with boundary such that there is only one fixed component~$B$ as a maximal fixed set with respect to the moment map with any given~$e_{\max}$.

	To construct non-extremal critical pieces containing fixed components of index~$2$, we need the following theorem due to Guillemin and Sternberg.
\begin{theorem} \label{GSthm} \cite[Section~12]{GS}
	Let $(M_0,\omega)$ be a  symplectic manifold, $X$ a symplectic submanifold of $M_0$, and $\pi :P \rightarrow M_0$ a principal $S^1$-bundle over $M_0$. Then for a given small open interval $I=(-\epsilon, \epsilon)$, there is a unique semifree Hamiltonian $S^1$-manifold $M_I$ with moment map $\Phi : M_I \rightarrow \R$ satisfying the following.

	\begin{itemize}
		\item[i)] $\Phi$ is proper and $\Phi(M_I) = I$.
		\item[ii)] For $-\tau<0 , \Phi^{-1}(-\tau)$ is $S^1$-equivariantly diffeomorphic to~$P$.
		\item[iii)]For $-\tau<0$, the reduction of $M_I$ at $-\tau$ is the manifold $M_0$ with symplectic form $\omega + \tau d\alpha$ where $\alpha$ is a connection 1-form on $P$.
		\item[iv)]$X$ is the only fixed set of index $2$ with respect to $\Phi$.
		\item[v)] For $\tau>0$, the reduction of $M_I$ at $\tau$ is the blow up, $M_+$ of $M_0$ along $X$ with symplectic form $\mu_{\tau} - \beta^* \tau d\alpha$ where $\mu_{\tau}$ is the $\tau$-blow up form and $\beta$ is the blowing down map.
	\end{itemize}
	Here, the $\tau$-blow up form $\mu_{\tau}$ is a symplectic form on blown-up space of $\tau$-amount.

\end{theorem}

	Therefore, if we have a symplectic manifold $(M_0,\omega)$, an integral cohomology class $e \in H^2(M_0,\Z)$, and a symplectic submanifold $X \subset M_0$, then we can construct a non-extremal critical piece $M$ with moment map $\mu : M \rightarrow (-\epsilon, \epsilon)$ such that $M_{\red} \cong M_0$, $e_- = e$ where $e_-$ is the Euler class of the principal $S^1$-bundle $\mu^{-1}(-\frac{\epsilon}{2})$ over $M_{\red}$, and $M^{S^1} = X$ of index 2. Moreover, such critical piece is unique up to $S^1$-equivariant symplectomorphism.

\begin{remark}
	Note that $M_I$ in Theorem~\ref{GSthm} is constructed from $P \times I$ by removing a small neighborhood $U$ of $\pi^{-1}(X) \times I$, and by gluing some open set along the boundary of $(P \times I) - U$. Moreover, this surgery does not change the symplectic structure outside of $U$. Hence $X$ need not be connected. For more details, see \cite{GS} Section~12.
\end{remark}

	It remains to glue these ``local pieces" along their boundaries. Let $(M,\omega)$ and $(N,\omega')$ be two local pieces with moment maps $\mu_M : M \rightarrow [a,b] \in \R$ and $\mu_N : N \rightarrow [b,c] \in \R$. Let $(M_b,\omega_b)$, respectively $(N_b,\omega_b')$, be a reduced space at $b$ with respect to $\mu_M$, respectively $\mu_N$. Assume that there is a symplectomorphism $\phi : (M_b,\omega_b) \rightarrow (N_b, \omega_b')$ such that $\phi^*e_N = e_M$ where $e_N$ is the Euler class of the principal $S^1$-bundle $\mu_N^{-1}(b)$ and $e_M$ is the one of $\mu_M^{-1}(b)$.

\begin{lemma}\cite[Lemma~13]{L2}\label{glue}
	The symplectomorphism $\phi$ induces an $S^1$-equivariant diffeomorphism $\widetilde{\phi} : \mu_M^{-1}(b) \rightarrow \mu_N^{-1}(b)$.
\end{lemma}

	In the case when the reduced spaces of local pieces are diffeomorphic to a ruled surface, we need the following theorem.

\begin{theorem} \label{LM} \cite{LM}
	Let $\omega_0$ and $\omega_1$ be two cohomologous compatible symplectic forms on a ruled surface. Then they are isotopic.
\end{theorem}

	Combining Theorem~\ref{LL} and Theorem~\ref{LM}, we have the well-known theorem which classifies symplectic structures on a ruled surface.

\begin{corollary}\label{combine}
	There is a symplectomorphism between two cohomologous symplectic forms on a ruled surface. In particular, we may choose a symplectomorphism which acts on homology trivially.
\end{corollary}
Hence if $M \cong N$ is a ruled surface and if $[\omega] = [\omega']$ and $e_N = e_M$, then we can glue $M$ and $N$ along $\mu^{-1}_M(b)$.

\begin{remark}
	Although what we need is existence, we can also think about uniqueness. By a result of Gonzalez~\cite{Gon}, a Hamiltonian $S^1$-manifold $(M, \omega)$ is uniquely determined by its fixed point data up to equivariant symplectomorphism when $(M_{\red}, \{\omega_t\})$ is rigid.
\end{remark}

\begin{definition} \cite{Gon}
	Let $B$ be a manifold and $\{\omega_t\}$ be a smooth family of symplectic forms on $B$. The pair $(B, \{\omega_t\})$ is said to be \textit{rigid} if
	\begin{itemize}
		\item[(i)] Symp$(B$, $\omega_t) \cap$ Diff$_0(B)$ is path connected for all $t$.
		\item[(ii)] Any deformation between any two cohomologous symplectic forms which are deformation equivalent to $\omega_t$ on $B$ may be homotoped through deformations with fixed endpoints into an isotopy.
	\end{itemize}
\end{definition}

	According to Proposition 1.6 in~\cite{MD4}, (i) holds for $(B = \Sigma_g \times S^2, \omega_t )$ when $\frac{a}{b} \geq [g/2]$, where $[\omega_t] = au + bv$. By Theorem 1.2 in~\cite{MD3}, (ii) holds for any ruled surfaces. Combining these results, we can see that the following examples with the given data are uniquely determined up to equivariant symplectomorphism.

\begin{example} \label{4genus}
	Let $(\Sigma_{\min}, \omega_{\min}) = (\Sigma_g, \sigma)$ with $g \geq 1$, where $\sigma$ is the normalized symplectic form on $\Sigma_g$. Consider a trivial $\C^2$-bundle $E_{\min}$ over $\Sigma_{\min}$. By Proposition~\ref{construction}, there is a symplectic form on $E$ and a Hamiltonian $S^1$-action on $E$ with moment map $\mu$ with $\mu(\Sigma_{\min}) = 0$. Then for a sufficiently small $\epsilon$, $\{x \in E_{\min} \mid \mu(x) \leq \epsilon\}$ gives a minimal critical piece, still denoted by $E_{\min}$. By Lemma \ref{Li}, we have the reduced space $M_{\red} = \Sigma_g \times S^2$ and $e_{\min} = -v$, where $\{u,v\}$ is the basis of $H^2(\Sigma_g \times S^2,\Z)$ chosen in Section 3. Therefore, the boundary of $E_{\min}$ is symplectomorphic to a principal $S^1$-bundle over $(\Sigma_g \times S^2, \omega_{\epsilon})$ whose Euler class is $-v$ and $[\omega_{\epsilon}] = u + \epsilon v$.

	Now, let $P$ be a principal $S^1$-bundle over $\Sigma_g \times S^2$ whose Euler class is $-v$, and let $\{ \omega_t = \sigma + t\tau \}_ {t \in [\epsilon, 1-\epsilon]}$ be a family of symplectic forms on $\Sigma_g \times S^2$ where $\tau$ is the normalized symplectic form on $S^2$. By Proposition \ref{regular}, we have a regular piece $E = (P \times [\epsilon, 1-\epsilon], \omega)$ such that $[\omega_t] = u + tv$ and it is well-glued to $E_{\min}$ by Corollary~\ref{combine}.

	For $(\Sigma_g \times S^2, \sigma + \tau)$, $P$, and a submanifold $X = \Sigma_g\times\{p_1, p_2\}$ for $p_1 \neq p_2$, we have a critical piece $E_{mid}$ for an open interval $(1-\epsilon, 1+\epsilon)$ by Theorem \ref{GSthm}. It is easy to check that this critical piece can be well-glued to the regular piece $E$.

	Note that $D(X) = 2v$. Hence, the Euler class of the principal $S^1$-bundle $\mu^{-1}(1+\frac{\epsilon}{2})$ is $v$. Let $Q$ be a principal $S^1$-bundle over $\Sigma_g \times S^2$ whose first Chern class is $v$, and let $\{ \omega_t = \sigma + (2-t)\tau \}_ {t \in [1+\epsilon, 2-\epsilon]}$ be a family of symplectic forms on $\Sigma_g \times S^2$. Again by Proposition \ref{regular}, we have a regular piece $E' = (P' \times [1+\epsilon, 2-\epsilon], \omega')$ such that $[\omega_t'] = u + (2-t)v$. Of course it is well-glued to $E_{mid}$ by Corollary \ref{combine}.

	Finally, consider a trivial $\C^2$-bundle over $(\Sigma_g, \sigma)$. By a similar construction of a minimal critical piece, except for considering a diagonal $S^1$-action on $\C^2$ with weight $(-1,-1)$, we have a maximal critical piece which can be well-glued to $E'$. This finishes the construction. We choose $K \geq g/2$ to guarantee uniqueness.

        \begin{displaymath}
            \begin{array}{lll}
                0 < t < 1, & [\omega_t] = Ku+tv, & e = -v.\\[0.5em]
                t=1, & X = \Sigma_g\times\{p_1, p_2\}, & D(X) = 2v.\\[0.5em]
                1 < t < 2, & [\omega_t] = Ku+(2-t)v, & e= v.\\[0.5em]
            \end{array}
        \end{displaymath}
        The above Hamiltonian semifree $S^1$-manifold has the fixed point set consisting of four surfaces of genus $g$, and $M_{\red} \cong \Sigma_g \times S^2$.

        Similarly, we can construct a 6-dimensional closed Hamiltonian semi-free $S^1$-manifold $M$ with the information below.
        In this case, $M$ has a fixed point set consisting of $N$ spheres  for any $N\geq 1$ and four surfaces of genus $g$. Here $M_{\red} = \Sigma_g \times S^2$ with $g\geq 1$. Fix a positive number $K \geq N + g$. Then
        \begin{displaymath}
            \begin{array}{lll}
                0 < t < 2, & [\omega_t] = Ku+tv, & e = -v.\\[0.5em]
                t=2, & X_1 = \Sigma_g\times\{p_1, p_2\}, & D(X_1) = 2v.\\[0.5em]
                2 \leq t < 3, & [\omega_t] = Ku+(4-t)v, & e = v.\\[0.5em]
                t=3, & X_2 = \{q_1, \cdots, q_{N}\}\times S^2, & D(X_2) = Nu.\\[0.5em]
                3 \leq t < 4, & [\omega_t] = (K-(t-3)N)u+(4-t)v, & e= Nu + v.\\[0.5em]
            \end{array}
        \end{displaymath}

    \end{example}

    \begin{remark}
        Let $(M,\omega)$ be a 6-dimensional closed Hamiltonian semifree $S^1$-manifold whose fixed point set consists of $N$ surfaces.
        Then with the maximal and the minimal fixed surfaces, $N$ should be at least $2$.
        The examples with $N=2$ and $N=3$ are given in \cite{L1} and \cite{L2}.
    \end{remark}

    \begin{example}\label{sphere} This example shows that there is a 6-dimensional closed Hamiltonian semifree $S^1$-manifold $M$ whose fixed point set consists of $N$ spheres for arbitrary $N \geq 4$, with $M_{\red} = S^2 \times S^2$.
        \begin{displaymath}
            \begin{array}{lll}
                0 < t < 2, & [\omega_t] = (N+1)u+tv, & e = -v.\\[0.5em]
                t=2, & X_1 = S^2 \times\{p_1, p_2\}, & D(X_1) = 2v.\\[0.5em]
                2 \leq t < 3, & [\omega_t] = (N+1)u+(4-t)v, & e = v.\\[0.5em]
                t=3, & X_2 = \{q_1, \cdots, q_{N}\}\times S^2, & D(X_2) = Nu.\\[0.5em]
                3 \leq t < 4, & [\omega_t] = ((4-t)N+1)u+(4-t)v, & e= Nu + v.\\[0.5em]
            \end{array}
        \end{displaymath}

    \end{example}

\end{document}